\documentclass{amsart}

\usepackage{mathrsfs}
\usepackage{amscd}
\usepackage{amsmath}
\usepackage{amssymb}
\usepackage{amsthm}
\usepackage{epsf}
\usepackage{latexsym}
\usepackage{verbatim}
\usepackage[all, cmtip]{xy}
\usepackage{tikz}
\usetikzlibrary{calc}
\usetikzlibrary{positioning}
\usetikzlibrary{matrix}
\usepackage{float}
\usepackage[hidelinks]{hyperref}
\usepackage{comment}
\usepackage{enumitem}
\usepackage{subcaption}
\usepackage{tikz-cd}
\usepackage[capitalise]{cleveref}

\tikzstyle{bsq}=[rectangle, draw, thick, minimum width=.5cm, minimum height=.5cm]
\tikzstyle{bver}=[rectangle, draw, thick, minimum width=1cm, minimum height=2cm]
\tikzstyle{bhor}=[rectangle, draw, thick, minimum width=2cm, minimum height=1cm]

\tikzstyle{divisor}=[circle,very thick,draw,scale=0.3,fill=white]
\tikzstyle{vertex}=[circle,draw,scale=0.3,fill=black]
\usetikzlibrary{patterns}

\usepackage[left=3.2cm, right=3.2cm]{geometry}

\newtheorem{theorem}{Theorem}[section]
\newtheorem{lemma}[theorem]{Lemma}
\newtheorem{conjecture}[theorem]{Conjecture}
\newtheorem{corollary}[theorem]{Corollary}
\newtheorem{proposition}[theorem]{Proposition}

\newtheorem{varexample}[theorem]{Example}

\theoremstyle{definition}
\newtheorem{remark}[theorem]{Remark}

\newtheorem{definition}[theorem]{Definition}

\def\PP{{\mathbb P}}

\def\cM{\mathcal{M}}

\def\cE{\mathcal{E}}

\newcommand{\cL}{\mathcal{L}}
\newcommand{\cO}{\mathcal{O}}

\newcommand{\cF}{\mathcal{F}}

\newcommand{\Z}{\mathbb{Z}}
\newcommand{\R}{\mathbb{R}}

\newcommand{\ord}{\operatorname{ord}}
\newcommand{\Trop}{\operatorname{Trop}}
\newcommand{\trop}{\operatorname{trop}}
\newcommand{\ddiv}{\operatorname{div}}
\newcommand{\Div}{\operatorname{Div}}

\newcommand{\PL}{\operatorname{PL}}

\newcommand{\Pic}{\operatorname{Pic}}

\newcommand{\rk}{\operatorname{rk}}

\newcommand{\WG}{\widetilde{\Gamma}}
\newcommand{\WC}{\widetilde{C}}

\newcommand{\Mod}[1]{\ (\mathrm{mod}\ #1)}

\newcommand{\an}{\mathrm{an}}

\newenvironment{example}{\begin{varexample}
\begin{normalfont}}{\end{normalfont}
\end{varexample}}

\begin{document}

\title{Prym-Brill-Noether Theory for General Covers}
\author{David Jensen}

\bibliographystyle{alpha}

\begin{abstract}
We bound the dimension of the Prym-Brill-Noether variety for an open subset of the moduli space of \'{e}tale double covers of $k$-elliptic curves.  We also obtain new bounds on the dimension of the Prym-Brill-Noether variety for general \'{e}tale double covers of $k$-gonal curves, disproving a conjecture of Creech, Len, Ritter, and Wu, and provide a new conjecture for its dimension.  To do this, we completely describe the Prym-Brill-Noether variety of a double cover of a certain tropical curve known as the loop of loops.  We use the combinatorics of Coxeter groups to establish several topological properties of these tropical Prym-Brill-Noether varieties, and prove a lifting result when the edge lengths are generic.
\end{abstract}

\maketitle

\setcounter{tocdepth}{1}

\section{Introduction}
\label{Sec:Intro}

\subsection{Prym-Brill-Noether Theory}
\label{Sec:IntroPBN}

Given an \'{e}tale double cover of curves $f \colon \WC \to C$, one defines the set $P(C,f)$ of \emph{Prym divisors} to be the set of divisor classes with canonical norm.  That is:
\[
P (C,f) := \{ [D] \in \Pic (\WC) \mid \mathrm{Nm}_f [D] = [K_C] \} .
\]
The set $P(C,f)$ has two connected components, $P^0 (C,f)$ and $P^1 (C,f)$, where:
\[
P^r (C,f) = \{ [D] \in P(C,f) \mid \rk (D) \equiv r \Mod{2} \}.
\]
Each component is a translate of the Prym variety $\mathrm{Prym}(C,f)$, and is thus isomorphic to a principally polarized abelian variety of dimension $g-1$, where $g$ is the genus of $C$.  For a nonnegative integer $r$, one defines the \emph{Prym-Brill-Noether variety}:
\[
V^r (C,f) := \{ [D] \in P^r (C,f) \mid \rk (D) \geq r \} .
\]

There is an irreducible moduli space $\mathcal{R}_g$ parameterizing \'{e}tale double covers $f \colon \WC \to C$, where the target curve has genus $g$.  In \cite{Welters85, Bertram87}, Bertram and Welters compute the dimension of $V^r(C,f)$ for a general element of $\mathcal{R}_g$, showing that
\[
\mathrm{dim} V^r (C,f) = g-1-{{r+1}\choose{2}}.
\]

More recently, Len and Ulirsch initiated the study of Prym-Brill-Noether theory for general covers \cite{LenUlirsch21}.  More precisely, for $k \geq 3$, there exists an irreducible moduli space of chains $\WC \xrightarrow{f} C \xrightarrow{h} \PP^1$, where $f$ is an \'{e}tale double cover, $C$ has genus $g$, and $h$ has degree $k$ \cite{BF86}.  The image of this moduli space in $\mathcal{R}_g$ is known as the $k$-\emph{gonal} locus.  Let $\ell = \lceil \frac{k}{2} \rceil$ and define:
\[
n(r,k) := {{\ell+1}\choose{2}} + \ell (r-\ell) .
\]
For a general chain in this moduli space, \cite{LenUlirsch21, CLRW20} show that
\[
\mathrm{dim} V^r (C,f) \leq \left\{ \begin{array}{ll}
g-1-n(r,k) &\text{if } \ell \leq r-1\\
g-1- {{r+1}\choose{2}} &\text{if } \ell \geq r.
\end{array} \right.
\]

In this note, we give a new proof of the theorem of Bertram and Welters (see Corollary~\ref{Cor:Welters}) and an improvement on the bound from \cite{LenUlirsch21, CLRW20} (see Theorem~\ref{Thm:KGonal} below).  In particular, we disprove \cite[Conjecture~3.9]{CLRW20}.  Our main result, however, is a bound on the dimension of the Prym-Brill-Noether variety for covers of elliptic curves.  We say that a curve $C$ is $k$-\emph{elliptic} if there exists a degree-$k$ map from $C$ to a curve of genus 1 that does not factor through an isogeny.  The $k$-elliptic locus in $\cM_g$ is a closed subvariety of dimension $2g-2$.  We define the $k$-elliptic locus in $\mathcal{R}_g$ to be the preimage of the $k$-elliptic locus in $\cM_g$.  In other words, an \'{e}tale double cover of curves $f \colon \WC \to C$ is $k$-elliptic if $C$ is $k$-elliptic.  When $k=2$, the $k$-elliptic locus in $\mathcal{R}_g$ has multiple components.  When $k \geq 2$, the $k$-elliptic locus in $\mathcal{R}_g$ is not known to be irreducible.  Our main result is the following.

\begin{theorem}
\label{Thm:MainThm}
Let $r \geq 1$ and $k \geq 2$.  There is a nonempty open subset of the $k$-elliptic locus in $\mathcal{R}_g$ such that, for every etale double cover $f \colon \WC\to C$ in this open subset, we have
\[
\mathrm{dim} V^r (C,f)  \leq \left\{ \begin{array}{ll}
g-1- \frac{k(r+1)}{2} &\text{if }  k < r\\
g-1- \frac{r(r+1)}{2} &\text{if } k \geq r,
\end{array} \right.
\]
with equality when $k \geq r$.
\end{theorem}

\noindent Theorem~\ref{Thm:MainThm} holds in characteristic zero, or when the characteristic is relatively prime to $2k$.

\subsection{Techniques and Further Results}
\label{Sec:Techniques}

Our approach is via tropical geometry.  In \cite{LenUlirsch21}, Len and Ulirsch construct a tropical analogue of the Prym-Brill-Noether variety.  They then compute the dimension of this variety for a double cover of a particular tropical curve, known as the \emph{chain of loops}.  The chain of loops has appeared in a number of research articles, and been used to great effect to prove results in Brill-Noether theory and related areas \cite{tropicalBN, tropicalGP, MRC, MRC2, JensenRanganathan17, CookPowellJensen22, CookPowellJensen22b, M23, R13}.

In this paper, we use a different tropical curve, known as the \emph{loop of loops}.  To our knowledge, the first appearance of this tropical curve in the literature is in \cite{LPP12}, where it is used as a pathological example of a tropical curve whose Brill-Noether theory is poorly behaved.  In subsequent appearances, it is typically mentioned only to highlight this pathological property \cite{Len16, TLS, Coppens25}.    As we will show, however, the Prym-Brill-Noether theory of the loop of loops is far more elegant.  Indeed, in Section~\ref{Sec:PrymDivisors} we show that, if $\pi \colon \WG \to \Gamma$ is a certain double cover of a loop of loops, then every divisor class $[D] \in P(\Gamma,\pi)$ has a canonical choice of representative $D \in \Div (\WG)$, giving a natural system of coordinates on $P(\Gamma,\pi)$.

In Section~\ref{Sec:Generic}, we completely describe the tropical Prym-Brill-Noether variety $V^r (\Gamma,\pi)$ in the case where $\Gamma$ has generic edge lengths.  This is partly a warmup for Section~\ref{Sec:Arbitrary}, where we describe $V^r (\Gamma,\pi)$ in the case where $\Gamma$ has arbitrary edge lengths.  Our description uses the combinatorics of Coxeter groups.  We define a \emph{lingering word} in a Coxeter group to be a word in which some of the terms may be the identity element in the group (see Section~\ref{Sec:Coxeter} for a precise definition).  In Theorem~\ref{Thm:Generic}, we show that, when $\Gamma$ has generic edge lengths, the Prym-Brill-Noether variety $V^r (\Gamma,\pi)$ is a union of tori, indexed by lingering words in the Coxeter group $A_r$ that contain a reduced subword for $w_0$, the maximal element with respect to Bruhat order.

The containment relations between these tori are encoded in a certain poset that we call the \emph{lingering subword poset}.  Combinatorial properties of this poset are explored in Section~\ref{Sec:Tropology}.  As consequences, we see that $V^r (\Gamma,\pi)$ is pure dimensional (Corollary~\ref{Cor:PureDimension}) and connected in codimension 1 (Corollary~\ref{Cor:Connected}).  When $g-1 = {{r+1}\choose{2}}$, the Prym-Brill-Noether variety $V^r (\Gamma,\pi)$ is finite, and its cardinality is equal to the number of reduced words for $w_0$ in the Coxeter group $A_r$.  In 1982, Stanley conjectured this number to be equal to the number of standard Young tableaux on the ``staircase'' partition $\Delta_r = r + (r-1) + \cdots + 1$.  This was subsequently proved by Stanley via an algebraic proof \cite{Stanley84}, and by Edelman and Greene via an explicit bijection \cite{EdelmanGreene87}.  This is also known to be the cardinality of $V^r(C,f)$ for a general \'{e}tale double cover $f \colon \WC \to C$ \cite{DCP}.  In Section~\ref{Sec:Lifting}, we leverage this to prove the following lifting theorem.

\begin{theorem}
\label{Thm:Lifting}
Let $\Gamma$ be an $r$-generic loop of loops of genus $g$ and $\pi \colon \WG \to \Gamma$ the quotient by the antipodal involution.  Let $K$ be an algebraically closed nonarchimedean field, let $C$ be a curve of genus $g$ over $K$, and $f \colon \WC \to C$ an \'{e}tale double cover that specializes to $\pi$.  Then every divisor class in $V^r (\Gamma,\pi)$ lifts to a divisor class in $V^r (C,f)$.
\end{theorem}

The proof of Theorem~\ref{Thm:Lifting} follows the same rough outline as that of the lifting result in \cite{CJP}.  Specifically, we use the description of $V^r (C,f)$ as a type-D degeneracy locus to prove that it is a local complete intersection.  We then apply combinatorial properties of $V^r (\Gamma,\pi)$ to prove an analogous tropical statement, and invoke Rabinoff's lifting theorem \cite{Rabinoff12} to complete the argument.  As consequences of Theorem~\ref{Thm:Lifting}, we obtain a new proof of the result of Welters and Bertram (Corollary~\ref{Cor:Welters}) as well as the cases where equality holds in Theorems~\ref{Thm:MainThm} and~\ref{Thm:KGonal}.  We also see that $V^r (C,f)$ is reduced (Corollary~\ref{Cor:Reduced}) and that every divisor class in $V^r (\Gamma,\pi)$ supports a tropical linear series of rank $r$ (Corollary~\ref{Cor:TLS}).

Finally, in Section~\ref{Sec:Arbitrary}, we describe $V^r (\Gamma,\pi)$ when $\Gamma$ has arbitrary edge lengths.  In Theorem~\ref{Thm:Arbitrary}, we again see that $V^r (\Gamma,\pi)$ is a union of tori, indexed by certain types of words in the Coxeter group $A_r$.  In Section~\ref{Sec:kUniform}, we consider the special case where $\Gamma$ is $k$-\emph{uniform} and prove Theorem~\ref{Thm:MainThm}.  The ``bielliptic'' case where $k=2$ is particularly interesting.  In Theorem~\ref{Thm:Bielliptic}, we show that in this case, the Prym-Brill-Noether variety $V^r (\Gamma,\pi)$ is a union of tori, indexed by lingering words in the Coxeter group $I_2 (r+1)$ that contain a reduced subword for $w_0$.  Because of this, in this case the Prym-Brill-Noether variety satisfies many of the same properties as in the case of generic edge lengths.  It is pure dimensional (Proposition~\ref{Prop:BiellipticPureDimension}) and connected in codimension 1 (Proposition~\ref{Prop:BiellipticConnected}).  When $g-1 = r+1$, the tropical Prym-Brill-Noether variety $V^r (\Gamma,\pi)$ is finite, of cardinality 2 (Lemma~\ref{Lem:BiellipticCardinality}).  In future work, we prove a lifting result for the bielliptic loop of loops analogous to Theorem~\ref{Thm:Lifting}, showing that the inequality in Theorem~\ref{Thm:MainThm} is an equality when $k=2$ \cite{ABDJLNRRL}.

\subsection{Covers of $\PP^1$}
\label{Sec:IntroGonal}

We now discuss the Prym-Brill-Noether theory of general $k$-gonal curves.  In this setting, we obtain the following improvement on the previously known upper bound for $\mathrm{dim} V^r (C,f)$.

\begin{theorem}
\label{Thm:KGonal}
Let $r \geq 1$ and $k \geq 3$.  For a general \'{e}tale double cover $f \colon \WC \to C$ in the $k$-gonal locus in $\mathcal{R}_g$, we have
\[
\mathrm{dim} V^r (C,f)  \leq \left\{ \begin{array}{ll}
g-1-n(r,k) &\text{if } \ell(\ell + 1) \leq r+1\\
g-1- \frac{(k-1)(r+1)}{2} &\text{if } k+1 \leq r+1 < \ell (\ell + 1) \\
g-1- {{r+1}\choose{2}} &\text{if } k \geq r+1,
\end{array} \right.
\]
with equality when $k \geq r+1$.
\end{theorem}

As with Theorem~\ref{Thm:MainThm}, Theorem~\ref{Thm:KGonal} holds in characteristic zero, or when the characteristic is relatively prime to $2k$.  Unfortunately, this upper bound is not optimal in general.  Indeed, our expectation is the following.

\begin{conjecture}
\label{Conj:KGonal}
Let $r \geq 1$ and $k \geq 3$.  For a general \'{e}tale double cover $f \colon \WC \to C$ in the $k$-gonal locus in $\mathcal{R}_g$, we have
\[
\mathrm{dim} V^r (C,f)  = \left\{ \begin{array}{ll}
g-1- (k-1)(r+1) + {{k}\choose{2}} &\text{if } k < r+1 \\
g-1- {{r+1}\choose{2}} &\text{if } k \geq r+1.
\end{array} \right.
\]
\end{conjecture}

We now explain the rationale behind Conjecture~\ref{Conj:KGonal}.  The idea is similar to that pioneered by H. Larson in \cite{Larson21}.  Given a curve $C$ with a fixed map to $\PP^1$, Larson defines the splitting type locus to be the set of line bundles on $C$ whose pushforward to $\PP^1$ is isomorphic to a given vector bundle on $\PP^1$.  She then proves that, for a general element of the Hurwitz space, the splitting type loci are smooth of the expected dimension, and satisfy certain natural containment relations.  In our setting, we will see that the pushforward of a line bundle with canonical norm from $\WC$ to $\PP^1$ is an orthogonal bundle, and we analogously define an orthogonal splitting type locus.

By \cite{Mumford71}, if $f \colon \WC \to C$ is an \'{e}tale double cover and $\cL \in \Pic (\WC)$ has canonical norm, then there exists a nondegenerate symmetric bilinear form $f_* \cL \otimes f_* \cL \to K_C$.  If $h \colon C \to \PP^1$ is a map of degree $k$, then define $\cE = h_* f_* \cL$.  Pushing forward, we obtain a nondegenerate symmetric bilinear form $\cE \otimes \cE \to h_* K_C$.  The vector bundle $h_* K_C$ has a unique summand isomorphic to $\cO_{\PP^1} (-2)$.  Composing with projection onto this summand, we obtain a nondegenerate symmetric bilinear form
\[
\cE \otimes \cE \to \cO_{\PP^1} (-2) .
\]
In other words, $\cE$ is an $\cO_{\PP^1} (-2)$-valued orthogonal bundle.  Equivalently, $\cE \otimes \cO_{\PP^1} (1)$ is an $\cO_{\PP^1}$-valued orthogonal bundle.  Any $\cO$-valued orthogonal bundle on $\PP^1$ is self-dual.  In other words, there exists a vector $\vec{a} = (a_1 , \ldots , a_k) \in \Z_{\geq 0}^k$ such that 
\begin{align}
\label{eq:split}
\cE \otimes \cO_{\PP^1} (1) \cong \cO_{\PP^1} (\vec{a}) \oplus \cO_{\PP^1} (-\vec{a}).
\end{align}

A theorem of Grothendieck \cite{Grothendieck57} says that, conversely, every vector bundle $\cF$ on $\PP^1$ such that $\cF^* \cong \cF$ admits a unique quadratic form $q \colon \cF \to \cF^*$ making it an $\cO$-valued orthogonal bundle.  Under the splitting \eqref{eq:split}, the quadratic form $q$ is given by the matrix
\[
q = \left( \begin{array}{ll}
0 & \mathrm{Id} \\
\mathrm{Id} & 0
\end{array} \right).
\]
The space of first-order deformations of an orthogonal bundle $\cF$, as described in \cite{Hulek81}, is $H^1 (\mathrm{End}^q (\cF))$, where $\mathrm{End}^q (\cF)$ is the sheaf defined by
\[
\mathrm{End}^q (\cF) (U) := \{ s \in \mathrm{End} (\cF) (U) \mid s^{T} q + qs = 0 \}.
\]
When $a_1 , \ldots , a_k$ are written in increasing order, one computes:
\[
h^1 (\mathrm{End}^q ( \cO_{\PP^1} (\vec{a}) \oplus \cO_{\PP^1} (-\vec{a}) )) = \sum_{i > j} \Big( \max \{ 0, a_i + a_j -1 \} + \max \{ 0, a_i - a_j - 1 \} \Big).
\]

For $\vec{a} = (a_1 , \ldots , a_k) \in \Z_{\geq 0}^k$, we define the \emph{orthogonal splitting type locus}
\[
U^{\vec{a}} (C,f,h) = \{ \cL \in P(C,f) \mid h_* f_* \cL \otimes \cO_{\PP^1} (1) \cong \cO_{\PP^1} (\vec{a}) \oplus \cO_{\PP^1} (-\vec{a}) \} .
\]
Conjecture~\ref{Conj:KGonal} is then implied by the following more general conjecture.

\begin{conjecture}
\label{Conj:OrthogonalSplitting}
Let $r \geq 1$ and $k \geq 3$.   For a general chain $\WC \xrightarrow{f} C \xrightarrow{h} \PP^1$, where $f$ is an \'{e}tale double cover, $C$ has genus $g$, and $h$ has degree $k$, we have
\begin{align*}
\mathrm{dim} U^{\vec{a}} (C,f,h) &= g-1-h^1(\mathrm{End}^q ( \cO_{\PP^1} (\vec{a}) \oplus \cO_{\PP^1} (-\vec{a})) \\
&= g-1 - \sum_{i > j} \Big( \max \{ 0, a_i + a_j -1 \} + \max \{ 0, a_i - a_j - 1 \} \Big).
\end{align*}
\end{conjecture}

To see why Conjecture~\ref{Conj:OrthogonalSplitting} implies Conjecture~\ref{Conj:KGonal}, note that, if $\cE \otimes \cO_{\PP^1} (1) \cong \cO_{\PP^1} (\vec{a}) \oplus \cO_{\PP^1} (-\vec{a})$, then
\[
h^0 (\cL) = h^0 (\cE) = \sum_{i=1}^k a_i .
\]
Now, consider the vector $\vec{a}$ where $a_i$ is equal to either $\lfloor \frac{r+1}{k} \rfloor$ or $\lceil \frac{r+1}{k} \rceil$ for all $i$.  Across all vectors $\vec{a}$ with $\sum a_i \geq r+1$, this one is most balanced.  In this case, one obtains
\begin{align*}
\sum_{i > j} \max \{ 0, a_i + a_j -1 \} + \max \{ 0, a_i - a_j - 1 \} &= \sum_{i > j} \max \{ 0, a_i + a_j - 1 \} \\
&= (\# \{ i \mid a_i > 0 \} -1) \sum_{i=1}^k a_i - {{\# \{ i \mid a_i > 0 \} }\choose{2}} \\
&= \left\{ \begin{array}{ll}
(k-1)(r+1) - {{k}\choose{2}} &\text{if } k < r+1 \\
{{r+1}\choose{2}} &\text{if } k \geq r+1.
\end{array} \right.
\end{align*}

\subsection*{Acknowledgements}

This work was supported in part by NSF grant DMS–2451533, as well as the Simons Travel Support for Mathematicians program.  Thank you to Yoav Len and Sam Payne for comments on an early draft.  I'd also like to thank my daughter, Sophronia Jensen, for providing me with ample time to think about mathematics while I took her for regular 3-hour walks.  I'd also like to thank Ham's House (aka Kenwick Table) for hosting me and Sophie during my research time from 7 to 9 each morning.  It may have looked like I was just snuggling a baby, but I was also thinking about tropical Prym-Brill-Noether theory.

\section{Preliminaries}
\label{Sec:Prelim}

\subsection{Divisors on Metric Graphs}
\label{Sec:Divisors}

We briefly recall the theory of divisors on metric graphs.  For a more detailed study of these topics, we refer the reader to the survey \cite{BakerJensen}, or to the notes on the author's website:
\[
\mbox{\url{https://www.ms.uky.edu/~dhje223/MA 764 Spring 2025.html}.}
\]
A \emph{metric graph} $\Gamma$ is the metric space obtained from a finite graph $G$ by identifying each edge $e$ in $G$ with an interval of positive real length $\ell(e)$.  The finite graph $G$ is called a \emph{model} for $\Gamma$.  The \emph{genus} of a metric graph is its first Betti number.  In other words, the genus is $\#E(G) - \#V(G) + 1$, where $E(G)$ and $V(G)$ denote the set of edges and the set of vertices, respectively, of a model $G$.

The \emph{tangent directions} at a point $p$ in a metric graph $\Gamma$ are the germs of isometric embeddings of the closed interval $[0,\epsilon]$ into $\Gamma$ with 0 mapping to $p$.  We write $T_p (\Gamma)$ for the set of tangent directions at $p$ in $\Gamma$.  The number $\# T_p (\Gamma)$ is the \emph{valence} of $p$.  All but finitely many points of a metric graph $\Gamma$ have valence 2, and the vertex set of any model of $\Gamma$ contains all points of valence different from 2.  

The \emph{divisor group} $\Div (\Gamma)$ of a metric graph $\Gamma$ is the free abelian group on points of $\Gamma$.  A \emph{divisor} on $\Gamma$ is an element of this group -- that is, a formal $\Z$-linear combination of points of $\Gamma$.  The \emph{degree} of a divisor $D = \sum D(p) \cdot p$ is the sum of the coefficients $\deg (D) = \sum D(p)$.  A divisor $D$ is \emph{effective}, denoted $D \geq 0$, if $D(p) \geq 0$ for all $p \in \Gamma$.  We write $\mathrm{Eff}^d (\Gamma)$ for the set of all effective divisors of degree $d$ on $\Gamma$.

We write $\PL (\Gamma)$ for the set of piecewise-linear, continuous functions $\varphi \colon \Gamma \to \R$ with integer slopes.  The semifield $\PL (\Gamma)$ is the tropical analogue of the function field of an algebraic curve.  The \emph{order of vanishing} $\ord_p (\varphi)$ of a function $\varphi \in \PL (\Gamma)$ at a point $p \in \Gamma$ is the sum of the incoming slopes of $\varphi$ along the tangent directions at $p$.  Given a function $\varphi \in \PL (\Gamma)$, the corresponding \emph{principal divisor} is $\ddiv (\varphi) = \sum \ord_p (\varphi) \cdot p$.  Two functions $\varphi, \psi \in \PL (\Gamma)$ satisfy $\ddiv (\varphi) = \ddiv (\psi)$ if and only if there exists a real number $c \in \R$ such that $\varphi = \psi + c$.  In other words, if $D$ is principal, then by definition there exists a function $\varphi \in \PL (\Gamma)$ such that $\ddiv (\varphi) = D$, and this function is unique up to translation.

Two divisors $D, D' \in \Div (\Gamma)$ are \emph{linearly equivalent}, denoted $D \sim D'$, if $D-D'$ is principal.  We write $[D]$ for the linear equivalence class of a divisor $D$.  The \emph{Picard group} $\Pic (\Gamma)$ is the set of linear equivalence classes of divisors on $\Gamma$.  The degree of every principal divisor is zero, so the degree of a divisor is invariant under linear equivalence.  Given a divisor $D$, we write
\[
R(D) := \{ \varphi \in \PL (\Gamma) \mid \ddiv (\varphi) + D \geq 0 \} .
\]
The set $R(D)$ is a \emph{tropical module}, meaning that it is closed under pointwise minimum and addition of scalars.  We write $\vert D \vert = \{ \ddiv (\varphi) + D \mid \varphi \in R(D) \}$.

The \emph{Baker-Norine rank} of $D$ is defined as follows:
\[
\rk (D) = \max \{ r \mid D-E \text{ is equivalent to an effective divisor for all } E \in \mathrm{Eff}^r (\Gamma) \} .
\]
The \emph{canonical divisor} of a metric graph $\Gamma$ is the divisor $K_{\Gamma} = \sum (\mathrm{val} (p) - 2) \cdot p$.  If $\Gamma$ has genus $g$, then $\deg (K_{\Gamma}) = 2g-2$.  The Baker-Norine rank satisfies the Riemann-Roch formula:
\[
\rk (D) - \rk (K_{\Gamma} - D) = \deg (D) - g + 1 \text{ \cite{BakerNorine07, MikhalkinZharkov08}}.
\]
As a consequence, we see that  every divisor $D$ with $\deg (D) > 2g-2$ satisfies $\rk (D) = \deg (D) - g$.  Moreover, if $\deg (D) = g-1$, then $\rk (D) = \rk (K_{\Gamma} - D)$.

\begin{example}
\label{Ex:Loop}
An important case to which we will frequently return is that where the metric graph $\gamma$ is a single loop.  Let $w \in \gamma$ be a point, and let $D$ be a divisor on $\gamma$ of degree $s+1$.  Since $D-sw$ has degree 1, by the tropical Riemann-Roch theorem, there exists a point $x \in \gamma$ such that $D \sim sw+x$.

We show that the point $x$ is unique.  If $sw+x \sim sw+x'$, then $x \sim x'$, hence there exists a function $\varphi \in \PL (\gamma)$ such that $\ddiv (\varphi) = x-x'$.  If $x \neq x'$, then $\varphi$ is nonconstant.  The region where $\varphi$ obtains its maximum is then a union of points and closed intervals, each of which has two outgoing tangent directions.  It follows that the sum of the positive coefficients of $\ddiv (\varphi)$ is at least 2, a contradiction.

An important special case is when $D = sv + u$ for points $u,v \in \gamma$.  By the above, there exists a unique point $x \in \gamma$ such that $D \sim sw + x$.  By definition, the point $x$ is equal to $w$ if and only if there exists an integer $a$ and a function $\varphi \in \PL (\gamma)$ with slopes pictured as in Figure~\ref{Fig:Loop}.  If $\gamma$ has circumference $\ell$ and $\ell (v,w), \ell (w,u)$ denote the counterclockwise distances from $v$ to $w$ and from $w$ to $u$, respectively, then a function with the pictured slopes is continuous if and only if
\[
(s-a) \ell (v,w) = (a+1) \ell (w,u) + a (\ell - \ell (v,w) - \ell (w,u)).
\]
Combining like terms, we obtain:
\[
s \ell (v,w) = \ell (w,u) + a \ell .
\]
Therefore, we see that the point $x$ is equal to $w$ if and only if $\ell (w,u) \equiv s \ell (v,w)$ in $\R/\ell\Z$.

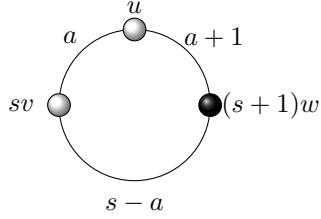
\begin{figure}[H]
\begin{tikzpicture}

\draw (0,0) circle (1);
\draw [ball color=white] (-1,0) circle (0.15);
\draw [ball color=black] (1,0) circle (0.15);
\draw [ball color=white] (0,1) circle (0.15);
\draw (-1.5,0) node {$sv$};
\draw (1.8,0) node {$(s+1)w$};
\draw (0,1.3) node {$u$};
\draw (1.05,0.9) node {$a+1$};
\draw (-0.85,0.9) node {$a$};
\draw (0,-1.3) node {$s-a$};

\end{tikzpicture}
\caption{Rightward slopes of a function $\varphi$ with $\ddiv (\varphi) = (s+1)w - (sv+u)$.}
\label{Fig:Loop}
\end{figure}
\end{example}

\subsection{Orientable Divisors}
\label{Sec:Orientable}

In this section, we discuss the connection between graph orientations and divisors on graphs.  While all the results we use can be found in earlier sources, the most comprehensive treatment of this theory is \cite{Backman17}.

An \emph{orientation} of a finite graph $G$ is a choice of head vertex and tail vertex for each edge.  Given an orientation $\cO$ of $G$ and a vertex $v \in V(G)$, the \emph{indegree} of $\cO$ at $v$, denoted $\mathrm{indeg}_{\cO}(v)$, is the number of edges adjacent to $v$ with head $v$.  An orientation of a metric graph $\Gamma$ is an orientation of some model for $\Gamma$.  Given an orientation $\cO$ of $\Gamma$, we define the corresponding \emph{orientable divisor}
\[
D_{\mathcal{O}} := \sum_{p \in \Gamma} (\mathrm{indeg}_{\mathcal{O}}(p) -1)p.
\]
If $\Gamma$ has genus $g$, then every orientable divisor on $\Gamma$ has degree $g-1$.  Given a point $p_0 \in \Gamma$, we say that an orientation $\cO$ is $p_0$-\emph{connected} if, for all $p \in \Gamma$, there is a directed path from $p_0$ to $p$.  In Section~\ref{Sec:DivisorsOnLOL}, we will use the following result.

\begin{proposition} \cite[Theorem~4.18]{ABKS}
\label{Prop:ConnectedOrientation}
Let $\Gamma$ be a metric graph of genus $g$ and let $p_0 \in \Gamma$.  For every divisor $D$ of degree $g-1$ on $\Gamma$, there exists a unique $p_0$-connected orientation $\cO$ such that $D \sim D_{\cO}$.
\end{proposition}

An orientation $\cO$ is \emph{acylic} if it does not contain a directed cycle.  The following result was first proved in \cite{BakerNorine07} for finite graphs, in the proof of the Riemman-Roch theorem for graphs, and then later for metric graphs in \cite{MikhalkinZharkov08}.  We note that \cite[Lemma~7.9]{MikhalkinZharkov08} does not explicitly state that the orientation $\cO$ is $p_0$-connected, but the orientation constructed in the proof has this property.

\begin{lemma} \cite[Lemma~7.9]{MikhalkinZharkov08}
\label{Lem:Acyclic}
Let $\Gamma$ be a metric graph and let $p_0 \in \Gamma$.  For every divisor $D$ on $\Gamma$, either $\rk (D) \geq 0$ or there exists a $p_0$-connected acyclic orientation $\cO$ such that $\rk(D_{\cO} - D) \geq 0$, but not both.
\end{lemma}

\begin{corollary}
\label{Cor:Acyclic}
Let $\Gamma$ be a metric graph, let $p_0 \in \Gamma$, and let $\cO$ be a $p_0$-connected orientation on $\Gamma$.  Then $\rk (D_{\cO}) = -1$ if and only if $\cO$ is acyclic.
\end{corollary}

\begin{proof}
If $\cO$ is acyclic, then since $\rk (D_{\cO} - D_{\cO}) = \rk(0) = 0$, by Lemma~\ref{Lem:Acyclic} we see that $\rk (D_{\cO}) = -1$.  Conversely, if $\rk (D_{\cO}) = -1$, then by Lemma~\ref{Lem:Acyclic}, there exists a $p_0$-connected acyclic orientation $\cO'$ such that $\rk (D_{\cO'} - D_{\cO}) \geq 0$.  Since $\deg (D_{\cO'} - D_{\cO}) = 0$ and the trivial divisor class is the only divisor class of degree zero with nonnegative rank, it follows that $D_{\cO'} - D_{\cO} \sim 0$, or $D_{\cO'} \sim D_{\cO}$.  By Proposition~\ref{Prop:ConnectedOrientation}, however, this implies that $\cO = \cO'$, hence $\cO$ is acyclic.
\end{proof}

\subsection{Tropical Covers and Prym Divisors}
\label{Sec:Covers}

In this section, we survey the results we will need from tropical Prym-Brill-Noether theory.  The reader is encouraged to check out \cite{JensenLen18, LenUlirsch21, GhoshZakharov} if they want to know more.

Let $\WG$ and $\Gamma$ be metric graphs.  A \emph{free double cover} is a degree 2 covering space $\pi \colon \WG \to \Gamma$ whose restriction to each closed interval in $\WG$ is an isometry onto its image.  If $\pi \colon \WG \to \Gamma$ is a free double cover, then there is a natural map $\mathrm{Nm}_{\pi} \colon \Pic (\WG) \to \Pic (\Gamma)$.  We define the set of \emph{Prym divisors} on $\WG$ to be:
\[
P (\Gamma,\pi) := \{ [D] \in \Pic (\WG) \mid \mathrm{Nm}_{\pi} [D] = [K_{\Gamma}] \}.
\]
The literature contains two distinct definitions of the tropical Prym variety \cite{JensenLen18, GhoshZakharov}.  For \emph{free} double covers, however, the two definitions coincide.  Since the only double covers we consider in this paper are free, we will not go into detail here.  We note, however, that the Prym varieties defined in \cite{GhoshZakharov} vary continuously in families, whereas the Prym varieties defines in \cite{JensenLen18} do not, so \cite{GhoshZakharov} is preferable.

If $\Gamma$ has genus $g$, then by \cite[Proposition~3.0.2]{LenUlirsch21}, $P (\Gamma,\pi)$ is the disjoint union of two real tori, each of dimension $g-1$.  For a nonnegative integer $r$, we define the \emph{tropical Prym-Brill-Noether variety}:
\[
V^r (\Gamma,\pi) := \{ [D] \in P (\Gamma,\pi) \mid \rk(D) \geq r \text{ and } \rk(D) \equiv r \Mod{2} \} .
\]

Nearly every map between metric graphs that appears in this paper is a free double cover.  Towards the end, however, we will need the more general theory of harmonic morphisms.  A continuous map $\pi \colon \WG \to \Gamma$ is a \emph{finite morphism} if there exist models $\widetilde{G}$ for $\WG$ and $G$ for $\Gamma$ such that the image of every vertex in $\widetilde{G}$ is a vertex in $G$, the preimage of every edge in $G$ is a set of edges in $\WG$, and the restriction of $\pi$ to any edge $\widetilde{e}$ of $\widetilde{G}$ is dilation by a positive integer $d_{\widetilde{e}} (\pi)$.

A finite morphism $\pi \colon \WG \to \Gamma$ is \emph{harmonic} at $\widetilde{p} \in \WG$ if, for every tangent vector $\eta \in T_{\pi (\widetilde{p})} (\Gamma)$, the sum
\[
d_{\widetilde{p}} (\pi) := \sum_{\widetilde{\eta} \in T_p (\WG), \pi (\widetilde{\eta}) =\eta} d_{\eta} (\pi)
\]
is independent of the choice of $\eta$.  The morphism $\pi$ is \emph{harmonic} if it is surjective and harmonic at every point $\widetilde{p} \in \WG$.  In this case, the integer $\deg (\pi) = \sum_{\widetilde{p} \in \pi^{-1} (p)} d_{\widetilde{p}} (\pi)$ is independent of the choice of point $p \in \Gamma$, and is called the \emph{degree} of the harmonic morphism.

If $\pi \colon \WG \to \Gamma$ is a harmonic morphism, then there are natural maps $\pi^* \colon \Div (\Gamma) \to \Div (\WG)$ and $\mathrm{Nm}_{\pi} \colon \Div (\WG) \to \Div (\Gamma)$, given by
\begin{align*}
\pi^* D &= \sum_{p \in \Gamma} \sum_{\widetilde{p} \in \pi^{-1} (p)} d_{\widetilde{p}} (\pi) D(\widetilde{p}) \widetilde{p} \\
\mathrm{Nm}_{\pi} \widetilde{D} &= \sum_{\widetilde{p} \in \WG} d_{\widetilde{p}} (\pi) D(\widetilde{p}) \pi (\widetilde{p}) .
\end{align*}
Both maps respect linear equivalence, and thus descend to maps between the corresponding Picard groups.

\subsection{Combinatorics of Coxeter Groups}
\label{Sec:Coxeter}

In this section, we review the basic combinatorics of Coxeter groups.  Everything in this section can be found in the textbook of Bjorner and Brenti \cite{BjornerBrenti}, except for the definition of lingering words, which is terminology specific to this paper.

Throughout, we let $(W,S)$ be a Coxeter group.  For a positive integer $n$, let $[n]$ denote the set of the first $n$ positive integers.  A \emph{word} $Q$ of \emph{length} $n$ in $(W,S)$ is a function $s \colon [n] \to S$, typically denoted $Q = s_1 s_2 \cdots s_n$.  We let $S^{\ast}$ denote the set of words in $(W,S)$.  A \emph{subword} of $Q$ is the restriction of $s$ to a subset of $[n]$.  A word $Q$ \emph{represents} an element $w \in W$ if $w$ is equal to the ordered product $\prod_{i=1}^n s_i$.  A word $Q$ of length $n$ is \emph{reduced} if no word of length less than $n$ represents the element $w = \prod_{i=1}^n s_i$.  The \emph{length} $\ell(w)$ of an element $w$ is the length of a reduced word that represents $w$.  We say that a word $Q$ \emph{contains a reduced word for} $w \in W$ if a subword of $Q$ represents $w$ and is reduced.

Let $\epsilon \in W$ be the identity element.  In this paper, we say that a \emph{lingering word} $Q$ of length $n$ is a function $s \colon [n] \to S \cup \{ \epsilon \}$.  The terminology is chosen to mirror that of ``lingering lattice paths'' from \cite{tropicalBN}.  Given a lingering word $Q =s_1 s_2 \cdots s_n$, a \emph{lingering subword} is a lingering word $Q' = s'_1 s'_2 \cdots s'_n$ of the same length such that, for all $j$, either $s'_j  = s_j$ or $s'_j = \epsilon$.  We say that a lingering word is \emph{reduced} if the subword consisting of non-identity elements is reduced.

The \emph{(strong) Bruhat order} is a partial order on $W$, defined as follows.  For $v, w \in W$, we write $v \leq w$ if some (equivalently, every) reduced word for $w$ contains a reduced word for $v$.  If $W$ is finite, there is a unique element $w_0 \in W$ that is maximal with respect to Bruhat order.  The element $w_0$ has order 2, so conjugation by $w_0$ defines an involution on $W$.  We write $\overline{w} = w_0 w w_0$ for the image of $w$ under this involution.

Given a word $Q$, there exists a unique element $\delta (Q)$ of $W$, known as the \emph{Demazure product} of $Q$, that is maximal among elements represented by subwords of $Q$ with respect to Bruhat order.  In other words, $Q$ contains a reduced subword for the Demazure product $\delta (Q)$, and if $Q$ contains a reduced word for $\sigma \in W$, then $\sigma \leq \delta (Q)$.

Most, though not all, of this article is focused on Coxeter groups of type A.  Let $S_{r+1}$ denote the symmetric group acting on the set $\{ 0, 1, \ldots , r \}$.  This indexing is somewhat nonstandard; in the literature on Coxeter groups, it is more typical for $S_{r+1}$ to act on the set $\{1, \ldots, r+1 \}$, but in our setting, the group will act on the terms in a ramification sequence, which are typically indexed starting with zero.  The group $S_{r+1}$ is generated by the simple transpositions $\tau_0 , \ldots , \tau_{r-1}$, where
\[
\tau_i (j) := \left\{ \begin{array}{ll}
i+1 &\text{if } j=i \\
i &\text{if } j=i+1 \\
j &\text{otherwise} . 
\end{array} \right.
\]
The Coxeter group $A_r$ is the group $W = S_{r+1}$ with generating set $S = \{ \tau_0 , \ldots \tau_{r-1} \}$.

An \emph{inversion} of a permutation $\sigma \in S_{r+1}$ is a pair $\{ i,j \} \subseteq \{ 0, \ldots , r \}$ such that $i<j$ and $\sigma (i) > \sigma (j)$.  The length of $\sigma$ is equal to the number of inversions.  More precisely, in a reduced word $Q = s_1 \cdots s_{\ell}$ for $\sigma$, each term $s_j$ inverts a pair of elements that were not inverted by the ordered product $\prod_{i=1}^{j-1} s_i$.

The permutation $w_0 \in S_{r+1}$ given by $w_0 (j) := r-j$ inverts every pair of elements, hence it is maximal with respect to Bruhat order.  Note that $\overline{\tau}_i = \tau_{r-1-i}$ for all $i$.  We write $w_1$ for the maximal element, with respect to Bruhat order, in the subgroup $S_r \subset S_{r+1}$ of permutations that fix $r$.  In other words,
\[
w_1 (j) := \left\{ \begin{array}{ll}
r-1-j &\text{if } j < r \\
r &\text{if } j=r . 
\end{array} \right.
\]
Note that $\overline{w}_1$ is the maximal element, with respect to Bruhat order, in the subgroup $S_r \subset S_{r+1}$ of permutations that fix 0.

For the Coxeter groups of type A, there is a simple, effective criterion for determining whether two permutations are comparable in Bruhat order.  Given a permutation $\sigma$, define the function
\[
\sigma [x,y] := \# \{ i \leq x \mid \sigma (i) \geq y \} .
\]
By \cite[Theorem~2.1.5]{BjornerBrenti}, one has $\sigma \leq \sigma'$ if and only if $\sigma [x,y] \leq \sigma' [x,y]$ for all $x, y \in \{ 0, \ldots , r \}$.

\section{Prym Divisors on a Loop of Loops}
\label{Sec:PrymDivisors}

\subsection{The Loop of Loops}
\label{Sec:LOL}

A \emph{loop of loops of genus $g$} is a metric graph $\Gamma$ with $2g-2$ vertices, labeled $w_j$ and $v_j$ for $j \in \Z/(g-1)\Z$.  For all $j$, there are two edges connecting $v_j$ to $w_j$ and one edge connecting $w_{j-1}$ to $v_j$, see Figure~\ref{Fig:LOL}.  We write $\gamma_j$ for the $j$th loop, consisting of the two edges connecting $v_j$ and $w_j$, and we write $\beta_j$ for the edge connecting $w_{j-1}$ to $v_j$.  We will always orient our figures so that the edge $\beta_1$ appears on the righthand side, and the loops $\gamma_j$ appear in order in the counterclockwise direction.

\begin{figure}[H]
\begin{tikzpicture}

\draw (0,0) circle (2);
\draw [fill=white] (0,2) circle (0.5);
\draw (0.65,2) node {$v_2$};
\draw (0,2) node{$\gamma_2$};
\draw [fill=white] (1.7,1) circle (0.5);
\draw (1.7,0.35) node {$v_1$};
\draw (1.5,1.6) node {$w_1$};
\draw (1.7,1) node{$\gamma_1$};
\draw [fill=white] (1.7,-1) circle (0.5);
\draw (1.7,-0.35) node {$w_6$};
\draw [fill=white] (0,-2) circle (0.5);
\draw [fill=white] (-1.7,-1) circle (0.5);
\draw [fill=white] (-1.7,1) circle (0.5);
\draw (2.2,0) node {$\beta_1$};
\draw (0.9,1.55) node {$\beta_2$};

\end{tikzpicture}
\caption{A loop of loops of genus 7, with several vertices, edges, and loops labeled.}
\label{Fig:LOL}
\end{figure}
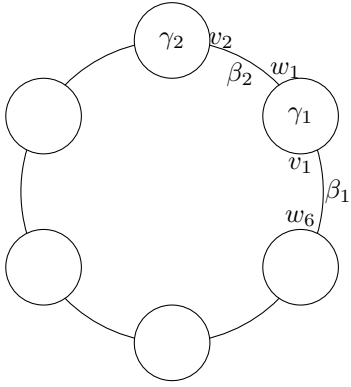

Following \cite{Pflueger17a}, we write $\ell (\beta_j)$ for the length of the edge $\beta_j$, $\ell (\gamma_j)$ for the length of $\gamma_j$, and $\ell (v_j,w_j)$ for the length of the clockwise edge from $v_j$ to $w_j$.  If $\ell (v_j,w_j)/\ell(\gamma_j)$ is irrational, we define the \emph{torsion order} $m_j$ of $\gamma_j$ to be 0.  Otherwise, we define the \emph{torsion order} of $\gamma_j$ to be the minimum positive integer $m_j$ such that $m_j \ell(v_j,w_j)$ is an integer multiple of $\ell(\gamma_j)$.  We refer to the sequence $\vec{m} = (m_1 , \ldots , m_{g-1})$ as the \emph{torsion profile} of $\Gamma$.  We say that the loop of loops $\Gamma$ is $r$-\emph{generic} if, for all $j$, either $m_j =0$ or $m_j \geq r$.  We say that the loop of loops $\Gamma$ is $k$-\emph{uniform} if $m_j = k$ for all $j$.  We note for future reference that the space of $r$-generic loops of loops has dimension $3g-3$, the same as that of the moduli space $\mathcal{M}_g$, and the space of $k$-uniform loops of loops has dimension $2g-2$, the same as that of the $k$-elliptic locus in $\mathcal{M}_g$.

The metric graph $\Gamma$ has a free double cover that is itself a loop of loops.  Throughout, we let $\WG$ denote a loop of loops of genus $2g-1$, with $\ell(\beta_j) = \ell(\beta_{g-1+j})$, $\ell(\gamma_j) = \ell(\gamma_{g-1+j})$, and $\ell(v_j,w_j) = \ell(v_{g-1+j},w_{g-1+j})$ for all $j \in \Z/(2g-2)\Z$.  Then $\WG$ admits an antipodal involution, and the quotient is a loop of loops $\Gamma$ of genus $g$.  For $p \in \WG$, we write $\overline{p}$ for its image under the antipodal involution.  See Figure~\ref{Fig:Involution}.  We write $\pi \colon \WG \to \Gamma$ for the quotient map.

\begin{figure}[H]
\begin{tikzpicture}

\draw (0,0) circle (2);
\draw [fill=white] (0,2) circle (0.5);
\draw [fill=white] (1.7,1) circle (0.5);
\draw [fill=white] (1.7,-1) circle (0.5);
\draw [fill=white] (0,-2) circle (0.5);
\draw [fill=white] (-1.7,-1) circle (0.5);
\draw [fill=white] (-1.7,1) circle (0.5);
\draw [ball color=black] (2,0) circle (0.55mm);
\draw [ball color=black] (-2,0) circle (0.55mm);
\draw (2.25,0) node {$p_0$};
\draw (-2.25,0) node {$\overline{p}_0$};
\draw [ball color=black] (0.35,2.35) circle (0.55mm);
\draw [ball color=black] (-0.35,-2.35) circle (0.55mm);
\draw (0.55,2.55) node {$p_1$};
\draw (-0.55,-2.55) node {$\overline{p}_1$};

\end{tikzpicture}
\caption{Some points in $\WG$ and their images under the antipodal involution.}
\label{Fig:Involution}
\end{figure}
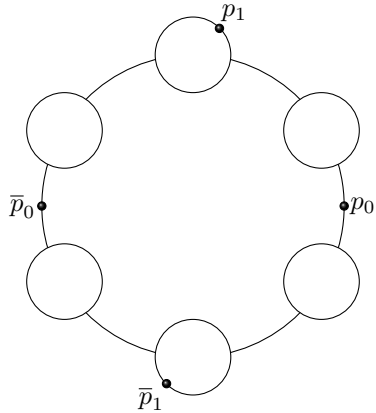

\subsection{Prym Divisors on the Loop of Loops}
\label{Sec:DivisorsOnLOL}

In this section, we completely describe the set of Prym divisor classes $P(\Gamma,\pi)$.  Again following \cite{Pflueger17a}, let $\langle \xi \rangle_j$ denote the point on $\gamma_j$ located $\xi \ell(v_j,w_j)$ units from $w_j$ in the counterclockwise direction.  Note that $\langle \xi \rangle_j = \langle \eta \rangle_j$ if and only if $\xi \equiv \eta$ in $\R/m_j\Z$.  By definition, $w_j = \langle 0 \rangle_j$ and $v_j = \langle -1 \rangle_j$.

We begin by defining a real torus of dimension $2g-2$:
\[
\widehat{\mathbb{T}}^0 := \Big\{ D \in \mathrm{Eff}^{2g-2} (\WG) \mid \deg (D_{\vert \gamma_j}) = 1 \text{ for all } j \} .
\]
For every divisor $D \in \widehat{\mathbb{T}}^0$, there exist unique elements $\xi_j (D) \in \R/m_j\Z$ such that
\[
D = \sum_{j=1}^{2g-2} \langle \xi_j (D) \rangle_j .
\]
We think of the functions $\xi_j$ as coordinates on the torus $\widehat{\mathbb{T}}^0$.  

Now, fix a basepoint $p_0 \in \WG$.  For simplicity, we will always assume that the basepoint $p_0$ is in the interior of the edge $\beta_1$, but we will see in Lemma~\ref{Lem:Bijection} below that our construction is independent of the choice of basepoint.  We define a second torus:
\[
\widehat{\mathbb{T}}^1 := \Big\{ D - p_0 + \overline{p}_0 \mid D \in \mathbb{T}^0 \} .
\]
A key observation is the following.

\begin{lemma}
\label{Lem:Orientable}
If $D \in \widehat{\mathbb{T}}^0 \cup \widehat{\mathbb{T}}^1$, then there exists a $p_0$-connected orientation $\cO$ of $\WG$ such that $D = D_{\cO}$.
\end{lemma}

\begin{proof}
We prove this by construction.  First, suppose that $D \in \widehat{\mathbb{T}}^0$.  Orient each edge $\beta_j$ from $w_{j-1}$ to $v_j$.  Next, there is at least one edge of $\gamma_j$ that does not contain $\langle \xi_j (D) \rangle_j$ in its interior.  Orient this edge from $v_j$ to $w_j$.  For the other edge of $\gamma_j$, orient the interval from $v_j$ to $\langle \xi_j (D) \rangle_j$ toward $\langle \xi_j (D) \rangle_j$, and the interval from $w_j$ to $\langle \xi_j (D) \rangle_j$ toward $\langle \xi_j (D) \rangle_j$.  The resulting orientation is $p_0$-connected because for all $j$ there is a directed path proceeding counterclockwise from $p_0$ to $v_j$, and the only point of $\gamma_j$ that could be a source of $\cO_{\vert \gamma_j}$ is $v_j$.  It is straightforward to check that $D = D_{\cO}$.  (See Figure~\ref{Fig:Orientation}.)

Next, suppose that $D \in \widehat{\mathbb{T}}^1$.  Then $D + p_0 - \overline{p}_0 \in \widehat{\mathbb{T}}^0$, hence by the above, there exists a $p_0$-connected orientation $\cO'$ of $\WG$ such that $D + p_0 - \overline{p}_0 = D_{\cO'}$.  Let $\cO$ be the orientation obtained from $\cO'$ by reversing a directed path from $\overline{p}_0$ to $p_0$ (again see Figure~\ref{Fig:Orientation}).  The resulting orientation is $p_0$-connected because for all $j \leq g-1$ there is a directed path proceeding counterclockwise from $p_0$ to $v_j$ and for all $j \geq g$ there is a directed path proceeding clockwise from $p_0$ to $w_j$, and again the restriction of the orientation to each loop $\gamma_j$ has at most one source.  Since $D + p_0 - \overline{p}_0 = D_{\cO'}$ and the indegree of $\cO$ differs from that of $\cO'$ only at $p_0$ and $\overline{p}_0$, we see that $D = D_{\cO}$.

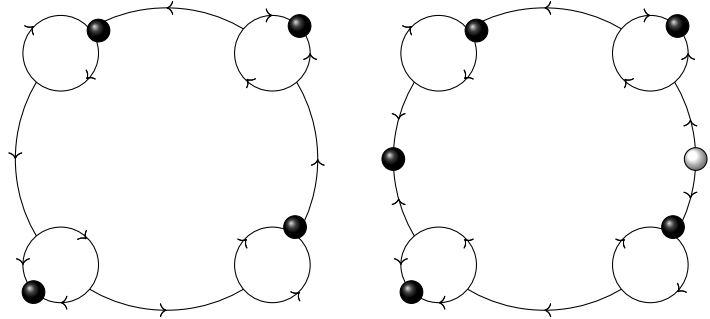
\begin{figure}[H]
\begin{tikzpicture}

\draw (0,0) circle (2);
\draw [fill=white] (1.4,1.4) circle (0.5);
\draw [fill=white] (1.4,-1.4) circle (0.5);
\draw [fill=white] (-1.4,-1.4) circle (0.5);
\draw [fill=white] (-1.4,1.4) circle (0.5);
\draw [ball color=black] (1.76,1.76) circle (0.15);
\draw [ball color=black] (-1.76,-1.76) circle (0.15);
\draw [ball color=black] (-0.9,1.7) circle (0.15);
\draw [ball color=black] (1.7,-0.9) circle (0.15);
\draw [<->] (2,-0.001)--(2,0);
\draw [<->] (0.001,2)--(0,2);
\draw [<->] (-2,0.001)--(-2,0);
\draw [<->] (-0.001,-2)--(0,-2);
\draw [<->] (1.061,1.059)--(1.06,1.06);
\draw [<->] (-1.059,1.061)--(-1.06,1.06);
\draw [<->] (-1.061,-1.059)--(-1.06,-1.06);
\draw [<->] (1.059,-1.061)--(1.06,-1.06);
\draw [<->] (-1.761,1.759)--(-1.76,1.76);
\draw [<->] (1.76,-1.76)--(1.761,-1.759);
\draw [<->] (1.9,1.399)--(1.9,1.4);
\draw [<->] (1.399,1.9)--(1.4,1.9);
\draw [<->] (-1.9,-1.399)--(-1.9,-1.4);
\draw [<->] (-1.399,-1.9)--(-1.4,-1.9);

\draw (5,0) circle (2);
\draw [fill=white] (6.4,1.4) circle (0.5);
\draw [fill=white] (6.4,-1.4) circle (0.5);
\draw [fill=white] (3.6,-1.4) circle (0.5);
\draw [fill=white] (3.6,1.4) circle (0.5);
\draw [ball color=black] (6.76,1.76) circle (0.15);
\draw [ball color=black] (3.24,-1.76) circle (0.15);
\draw [ball color=black] (4.1,1.7) circle (0.15);
\draw [ball color=black] (6.7,-0.9) circle (0.15);
\draw [ball color=white] (7,0) circle (0.15);
\draw [ball color=black] (3,0) circle (0.15);
\draw [<->] (5.001,2)--(5,2);
\draw [<->] (5.001,-2)--(5,-2);
\draw [<->] (6.061,1.059)--(6.06,1.06);
\draw [<->] (3.941,1.061)--(3.94,1.06);
\draw [<->] (3.941,-1.061)--(3.94,-1.06);
\draw [<->] (6.059,-1.061)--(6.06,-1.06);
\draw [<->] (3.239,1.759)--(3.24,1.76);
\draw [<->] (6.761,-1.759)--(6.76,-1.76);
\draw [<->] (6.93,0.519)--(6.93,0.52);
\draw [<->] (3.07,0.521)--(3.07,0.52);
\draw [<->] (6.93,-0.519)--(6.93,-0.52);
\draw [<->] (3.07,-0.521)--(3.07,-0.52);
\draw [<->] (6.9,1.399)--(6.9,1.4);
\draw [<->] (6.399,1.9)--(6.4,1.9);
\draw [<->] (3.1,-1.399)--(3.1,-1.4);
\draw [<->] (3.601,-1.9)--(3.6,-1.9);

\end{tikzpicture}
\caption{The two orientations described in the proof of Lemma~\ref{Lem:Orientable} for divisors in $\widehat{\mathbb{T}}^0$ (left) and divisors in $\widehat{\mathbb{T}}^1$ (right).}
\label{Fig:Orientation}
\end{figure}
\end{proof}

We write $\mathbb{T}^r$ for the image of $\widehat{\mathbb{T}}^r$ in $\Pic (\WG)$.  A consequence of Lemma~\ref{Lem:Orientable} is that every divisor class in $\mathbb{T}^r$ has a unique representative in $\widehat{\mathbb{T}}^r$.

\begin{corollary}
\label{Cor:2Tori}
The map from $\widehat{\mathbb{T}}^0 \sqcup \widehat{\mathbb{T}}^1$ to $\mathbb{T}^0 \cup \mathbb{T}^1$ is a homeomorphism.
\end{corollary}

\begin{proof}
It suffices to show that the map is injective.  This follows immediately from \cite[Theorem~4.18]{ABKS}, which says that two $p_0$-connected orientable divisors are linearly equivalent if and only if they are equal.
\end{proof}

By Corollary~\ref{Cor:2Tori}, the functions $\xi_j$ descend to well-defined coordinate functions on $\mathbb{T}^r$.  We now define two real subtori of dimension $g-1$:
\begin{align*}
P^0 &= \Big\{ [D] \in \mathbb{T}^0 \mid \xi_{g-1+j} (D) = -1 - \xi_j (D) \text{ for all } j \Big\} \\
P^1 &= \Big\{ [D] \in \mathbb{T}^1 \mid \xi_{g-1+j} (D) = -1 - \xi_j (D) \text{ for all } j \Big\}.
\end{align*}

The superscripts on these sets are to be understood as elements of $\Z/2\Z$, so $P^r$ is equal to $P^0$ when $r$ is even, and equal to $P^1$ when $r$ is odd.  By Corollary~\ref{Cor:2Tori}, both $P^0$ and $P^1$ are real tori of dimension $g-1$.  Specifically, each $P^r$ is homeomorphic to $\prod_{j=1}^{g-1} \R/m_j\Z$ via the coordinate functions $\xi_j$.

\begin{proposition}
\label{Prop:PrymDivisors}
We have
\(
P(\Gamma,\pi) = P^0 \sqcup P^1 .
\)
\end{proposition}

\begin{proof}
We first show that $P^0 \cup P^1 \subseteq P(\Gamma,\pi)$.  Let $[D]  \in P^0 \cup P^1$.  Then 
\begin{align*}
\mathrm{Nm}_{\pi} (D) &= \sum_{j=1}^{g-1} \langle \xi_j (D) \rangle_j + \xi_{g-1+j} (D) \rangle_j \\
&= \sum_{j=1}^{g-1} \langle \xi_j (D) \rangle_j + \langle -1 - \xi_j (D) \rangle_j \\
& \sim \sum_{j=1}^{g-1} \langle 0 \rangle_j + \langle -1 \rangle_j = \sum_{j=1}^{g-1} w_j + v_j = K_{\Gamma}.
\end{align*}
Hence, $[D] \in P(\Gamma,\pi)$.

Now, by \cite[Theorem~1.5.7]{LenUlirsch21}, $P(\Gamma,\pi)$ is a union of two $(g-1)$-dimensional real tori, and by Corollary~\ref{Cor:2Tori}, each of $P^0$ and $P^1$ is a $(g-1)$-dimensional real torus.  Since a connected manifold cannot properly contain a closed submanifold (without boundary) of the same dimension, it follows that $P(\Gamma,\pi) = P^0 \cup P^1$.

\end{proof}

The following lemma will often be used implicitly.

\begin{lemma}
\label{Lem:Bijection}
For any point $p \in \WG$, the map $F_p \colon P(\Gamma,\pi) \to P(\Gamma,\pi)$ given by $F_p ([D]) = [D-p+\overline{p}]$ is a bijection.  Under this bijection, the image of $P^0$ is $P^1$ and the image of $P^1$ is $P^0$.  In particular, $P^1$ is independent of the choice of basepoint $p_0$.
\end{lemma}

\begin{proof}
The inverse of $F_p$ is $F_{\overline{p}}$, so $F_p$ is a bijection.  The map $F_p$ is continuous, so it sends connected components to connected components.  Therefore, by Proposition~\ref{Prop:PrymDivisors}, the image of $P^0$ is either $P^0$ or $P^1$.  Now, for a fixed divisor class $[D] \in P^0$, the map from $\WG$ to $P(\Gamma,\pi)$ sending $p$ to $[D-p+\overline{p}]$ is continuous, hence its image is connected.  Since $[D-p_0+\overline{p}_0]$ is in $P^1$ by definition, it follows that $[D+p-\overline{p}]$ is in $P^1$ for all $p \in \WG$, and the result follows.
\end{proof}

In later sections, we will identify all divisors of rank $r$ in $P(\Gamma,\pi)$.  The proof will be by induction on $r$, and the following will form the base case for the induction.

\begin{lemma}
\label{Lem:BaseCase}
A divisor class $[D] \in P(\Gamma,\pi)$ satisfies $\rk(D) \geq 0$ if and only if either:
\begin{enumerate}
\item  $[D] \in P^0$, or
\item  $[D] \in P^1$, and $\xi_j (D) = -1$ for some $j \in \{ 1, 2, \ldots , g-1 \} \subset \Z/(2g-2)\Z$.
\end{enumerate}
\end{lemma}

\begin{proof}
By definition, every divisor in $\widehat{\mathbb{T}}^0$ is effective, hence every divisor class in $P^0$ has nonnegative rank.  Now, let $D \in P^1$.  We show that the orientation $\cO$ defined in the proof of Lemma~\ref{Lem:Orientable} is acyclic if and only if $\xi_j (D) \neq -1$ for all $ j\in \{ 1, 2, \ldots , g-1 \} \subset \Z/(2g-2)\Z$.  The result will then follow from Corollary~\ref{Cor:Acyclic}.

To see this, note that $p_0$ is a source in $\cO$.  It follows that no directed cycle in $\WG$ can contain $p_0$, hence $\WG$ contains a directed cycle if and only if $\cO_{\vert \gamma_j}$ is a directed cycle for some $j$.  By construction, if $\cO_{\vert \gamma_j}$ is a directed cycle for $j \in \{ 1, 2, \ldots , g-1 \}$, then $\xi_j (D) = -1$.  On the other hand, if $\cO_{\vert \gamma_j}$ is a directed cycle for $j \in \{ g , g+1, \ldots , 2g-2 \}$, then $\xi_j (D) = 0$, hence $\xi_{j-(g-1)} (D) = -1$.
\end{proof}

\section{The Prym-Brill-Noether Variety for $r$-Generic Loops of Loops}
\label{Sec:Generic}

\subsection{Prym Words}
\label{Sec:PrymWords}
Fix an integer $r \geq 1$.  Recall that a loop of loops $\Gamma$ is $r$-\emph{generic} if, for all $j$, either $m_j =0$ or $m_j \geq r$.  In this section, we let $\Gamma$ be an $r$-generic loop of loops, and identify the Prym-Brill-Noether variety $V^r(\Gamma,\pi)$.

First, we define the function $s^r_j [i] : \{ 0, \ldots , r \} \times \Z/(2g-2)\Z \to \Z$ by
\[
s^r_j [i] := \left\{ \begin{array}{ll}
i - \frac{r}{2} &\text{if } r \text{ is even }\\
i - \frac{r+1}{2} &\text{if } r \text{ is odd and } j \in \{ 1, \ldots , g-1 \} \subset \Z/(2g-2)\Z \\
i - \frac{r-1}{2} &\text{if } r \text{ is odd and } j \in \{ g, \ldots, 2g-2 \} \subset \Z/(2g-2)\Z . 
\end{array} \right.
\]
This notation is chosen to mirror \cite[Definition~9.2]{M23}, where $s_j [i]$ denotes the incoming slope at $v_j$ of a certain function in a tropical linear series on a chain of loops.  We will define analogous functions in Proposition~\ref{Prop:Fns} below, and the connection to tropical linear series will appear in Corollary~\ref{Cor:TLS}.  Note that, when $r$ is even, the values of $s^r_j$ are the $r+1$ consecutive integers $-\frac{r}{2} , \ldots , \frac{r}{2}$ centered at zero.  When $r$ is odd, the values are $-\frac{r+1}{2} , \ldots , \frac{r-1}{2}$ for $j \leq g-1$ and $-\frac{r-1}{2} , \ldots , \frac{r+1}{2}$ for $j \geq g$.  We similarly define
\[
s'^r_j [i]:= \left\{ \begin{array}{ll}
s^r_{j+1} [i] &\text{if } r \text{ is odd and } j \neq g-1, 2g-2 \text{ or } r \text{ is even}\\
s^r_{j+1} [i]+1 &\text{if } r \text{ is odd and } j = 2g-2 \\
s^r_{j+1} [i]-1 &\text{if } r \text{ is odd and } j = g-1 .
\end{array} \right.
\]
The integer $s'^r_j [i]$ denotes the outgoing slope at $w_j$ of a certain function on a loop of loops.

Let $(W,S)$ be the Coxeter group $A_r$.  To each divisor class $[D] \in P^r$, we associate a function $s \colon \Z/(2g-2)\Z \to S \cup \{ \epsilon \}$ as follows:
\[
s_j := \left\{ \begin{array}{ll}
\tau_i &\text{if } \xi_j (D) = s^r_j [i] \\
\epsilon &\text{otherwise.}
\end{array} \right.
\]
The assumption that $\Gamma$ is $r$-generic implies that, if $s^r_j [i] \equiv s^r_j [i'] \in \R/m_j\Z$ for some $i,i' \in \{0, \ldots , r-1 \}$, then $i=i'$.  It follows that $s$ is well-defined.  One should think of the function $s$ as a lingering word $Q_D$, indexed by $\Z/(2g-2)\Z$.  The lingering word $Q_D$ satisfies an important property.

\begin{definition}
\label{Def:PrymWord}
Let $(W,S)$ be a finite Coxeter group.  A \emph{Prym word} $Q$ of length $2g-2$ in $(W,S)$ is a function $s \colon \Z/(2g-2)\Z \to S \cup \{ \epsilon \}$ such that $s_{g-1+j} = \overline{s}_j$ for all $j \in \Z/(2g-2)\Z$. 
\end{definition}

\begin{lemma}
\label{Lem:PrymWord}
If $[D] \in P^r$, then $Q_D$ is a Prym word.
\end{lemma}

\begin{proof}
By definition, if $[D] \in P^r$, then $\xi_{g-1+j}(D) = -1 - \xi_j (D)$ for all $j$.  Thus, $\xi_j (D) = s^r_j [i]$ if and only if $\xi_{g-1+j}(D) = -1-s^r_j [i] = s^r_{g-1+j} [r-1-i]$, and the result follows.
\end{proof}

\begin{definition}
\label{Def:HalfWord}
Let $Q = s_1 s_2 \cdots s_{2g-2}$ be a Prym word.  For $j \in \Z/(2g-2)\Z$, the $j$th \emph{half word} of $Q$ is 
\[
Q^j = s_j s_{j+1} \cdots s_{g-2+j}.
\]
\end{definition}

The map $Q \mapsto Q^j$ sends a Prym word of length $2g-2$ to a lingering word of length $g-1$.  This map has an inverse.  If $Q = s_1 s_2 \cdots s_{g-1}$ is a lingering word of length $g-1$, we define $Q_j = s'_1 s'_2 \cdots s'_{2g-2}$ by 
\[
s'_i := \left\{ \begin{array}{ll}
s_{i+j-1} &\text{if } i \in \{ j, \ldots , j+g-2 \} \\
\overline{s}_{g-2+i+j} &\text{if } i \in \{ j+g-1, \ldots , j-1 \} .
\end{array} \right.
\]
Prym words of length $2g-2$ are therefore in bijection with lingering words of length $g-1$.

Another important property of Prym words is the following.

\begin{lemma}
\label{Lem:Containsw0}
In a Prym word, if some half word contains a reduced subword for $w_0$, then every half word contains a reduced subword for $w_0$.
\end{lemma}

\begin{proof}
Let $Q = s_1 \cdots s_{2g-2}$ be a Prym word, and suppose that $Q^j$ contains a reduced subword for $w_0$.  It suffices to show that $Q^{j+1}$ contains a reduced subword for $w_0$.  Let $s_{i_1} \cdots s_{i_\ell}$ be a subword of $Q^j$ that represents $w_0$.  If $i_1 \neq j$, then $s_{i_1} \cdots s_{i_\ell}$ is also a subword of $Q^{j+1}$, and the result follows.  If $i_1 = j$, then $Q^{j+1}$ contains the subword $s_{i_2} \cdots s_{i_\ell} s_{g-1+j}$.  The element of $W$ represented by this word is
\begin{align*}
s_{i_2} \cdots s_{i_\ell} s_{g-1+j} & = s_{i_2} \cdots s_{i_\ell} \overline{s}_j \\
&= s_{i_2} \cdots s_{i_\ell} w_0 s_j w_0 \\
&= s_j s_j s_{i_2} \cdots s_{i_\ell} w_0 s_j w_0 \\
&= s_j w_0^2 s_j w_0 = s_j^2 w_0 = w_0.
\end{align*}
\end{proof}

The Prym word of the Serre dual is easy to describe.

\begin{lemma}
\label{Lem:SerreDual}
For $[D] \in P (\Gamma,\pi)$, we have $K_{\WG} - D \sim \overline{D}$.  Moreover, if $Q_D = s_1 \cdots s_{2g-2}$, then $Q_{\overline{D}} = \overline{s}_1 \cdots \overline{s}_{2g-2}$. 
\end{lemma}

\begin{proof}
To see the first statement, note that
\[
D + \overline{D} = \pi^* \mathrm{Nm}_{\pi} D = \pi^* K_{\Gamma} = K_{\WG} .
\]
Now, if $[D] \in P^0$, then $\xi_j (\overline{D}) = \xi_{g-1+j} (D) = -1 - \xi_j (D)$ for all $j$.  If $[D] \in P^1$, then by Lemma~\ref{Lem:Bijection}, there exists a function $\varphi \in R(D)$, unique up to translation,  such that $\overline{D} + \ddiv (\varphi) \in \widehat{\mathbb{T}}^1$.  The function $\varphi$ has constant slope 1 along each edge $\beta_j$ for $j \in \{ 2, \ldots , g-1 \}$ and constant slope $-1$ along each edge $\beta_j$ for $j \in \{ g+1, \ldots, 2g-2 \}$.  From this it follows that $\xi_j (\overline{D}) = -1 + \xi_{g-1+j} (D) = -2 - \xi_j (D)$ for all $j \in \{ 1, \ldots , g-1 \}$ and $\xi_j (\overline{D}) = 1 + \xi_{g-1+j} (D) = - \xi_j (D)$ for all $j \in \{ g, \ldots , 2g-2 \}$.  In each case, we see that $\xi_j (D) = s^r_j [i]$ if and only if $\xi_j (\overline{D}) s^r_j [r-1-i]$, and the result follows.
\end{proof}

Given a lingering word $Q = s_1 \cdots s_{g-1}$ of length $g-1$ in the Coxeter group $A_r$, we define the set
\begin{align*}
T(Q) :&= \{ [D] \in P^r \mid Q \text{ is a lingering subword of } Q_D^1 \} \\
& = \{ [D] \in P^r \mid \xi_j (D) = s^r_j [i] \text{ whenever } s_j = \tau_i \} .
\end{align*}

The following is the main result of this section.

\begin{theorem}
\label{Thm:Generic}
Let $\Gamma$ be an $r$-generic loop of loops and $\pi \colon \WG \to \Gamma$ the quotient by the antipodal involution.  Then
\[
V^r (\Gamma,\pi) = \bigcup T(Q),
\]
where the union is over all lingering words of length $g-1$ in the Coxeter group $A_r$ that contain a reduced subword for $w_0$.
\end{theorem}

The proof of Theorem~\ref{Thm:Generic} will occupy the next two subsections.  Before proving this result, we record a few elementary observations about the sets $T(Q)$.

\begin{lemma}
\label{Lem:DimTQ}
Let $Q = s_1 \cdots s_{g-1}$ be a lingering word. Then $T(Q)$ is a real torus of dimension $\# \{ j \mid s_j = \epsilon \}$.
\end{lemma}

\begin{proof}
If $s_j = \tau_i$, then for all $[D] \in T(Q)$, we have $\xi_j (D) = s^r_j [i]$.  In other words, the coordinate function $\xi_j$ is fixed for all $[D] \in T(Q)$.  On the other hand, if $s_j = \epsilon$, then the coordinate function $\xi_j$ may vary freely.  It follows that $T(Q)$ is homeomorphic to $\prod_{s_j = \epsilon} \R/m_j\Z$.
\end{proof}

\begin{lemma}
\label{Lem:Contains}
Let $Q$ and $Q'$ be lingering words.  Then $T(Q) \subseteq T(Q')$ if and only if $Q'$ is a lingering subword of $Q$.
\end{lemma}

\begin{proof}
Let $[D] \in T(Q)$.  By definition, $Q$ is a lingering subword of $Q^1_D$, hence if $Q'$ is a lingering subword of $Q$, then $Q'$ is a lingering subword of $Q^1_D$, so $[D] \in T(Q')$.  For the converse, let $Q = s_1 \cdots s_{g-1}$ and $Q' = s'_1 \cdots s'_{g-1}$.  By definition, if $Q'$ is not a lingering subword of $Q$, then there exists $j$ such that $s'_j \neq \epsilon$, and $s_j \neq s'_j$.  Suppose that $s'_j = \tau_i$.  Since $s_j \neq \tau_i$, there exists a divisor $[D] \in T(Q)$ such that $\xi_j (D) \neq s^r_j [i]$, hence $[D] \notin T(Q')$.
\end{proof}

\begin{example}
\label{Ex:Rank2}
When $g-1 = {{r+1}\choose{2}}$, if $Q = s_1 \cdots s_{g-1}$ is a lingering word that contains a reduced subword $Q'$ for $w_0$, then $Q$ must be equal to $Q'$.  In particular, we see that $s_j \neq \epsilon$ for all $j$, hence $T(Q)$ consists of a single divisor class.

In the case where $r=2$, there are exactly two reduced subwords for $w_0$, corresponding to the two divisor classes pictured in Figure~\ref{Fig:Rank2}.  By Theorem~\ref{Thm:Generic}, $V^2 (\Gamma,\pi)$ consists of exactly these two divisor classes.

\begin{figure}[H]
\begin{tikzpicture}

\draw (0,0) circle (2);
\draw [fill=white] (0,2) circle (0.5);
\draw [fill=white] (1.7,1) circle (0.5);
\draw [fill=white] (1.7,-1) circle (0.5);
\draw [fill=white] (0,-2) circle (0.5);
\draw [fill=white] (-1.7,-1) circle (0.5);
\draw [fill=white] (-1.7,1) circle (0.5);
\draw [ball color=black] (1.93,0.52) circle (0.15);
\draw [ball color=black] (-0.52,1.93) circle (0.15);
\draw [ball color=black] (-1.4,1.4) circle (0.15);
\draw [ball color=black] (-1.4,-1.4) circle (0.15);
\draw [ball color=black] (-0.52,-1.93) circle (0.15);
\draw [ball color=black] (1.93,-0.52) circle (0.15);

\draw (5,0) circle (2);
\draw [fill=white] (5,2) circle (0.5);
\draw [fill=white] (6.7,1) circle (0.5);
\draw [fill=white] (6.7,-1) circle (0.5);
\draw [fill=white] (5,-2) circle (0.5);
\draw [fill=white] (3.3,-1) circle (0.5);
\draw [fill=white] (3.3,1) circle (0.5);
\draw [ball color=black] (3.07,0.52) circle (0.15);
\draw [ball color=black] (5.52,1.93) circle (0.15);
\draw [ball color=black] (6.4,1.4) circle (0.15);
\draw [ball color=black] (6.4,-1.4) circle (0.15);
\draw [ball color=black] (5.52,-1.93) circle (0.15);
\draw [ball color=black] (3.07,-0.52) circle (0.15);

\end{tikzpicture}
\caption{Two Prym divisors of rank 2 on a loop of loops of genus 7, corresponding to the words $\tau_0 \tau_1 \tau_0$ (left) and $\tau_1 \tau_0 \tau_1$ (right).}
\label{Fig:Rank2}
\end{figure}
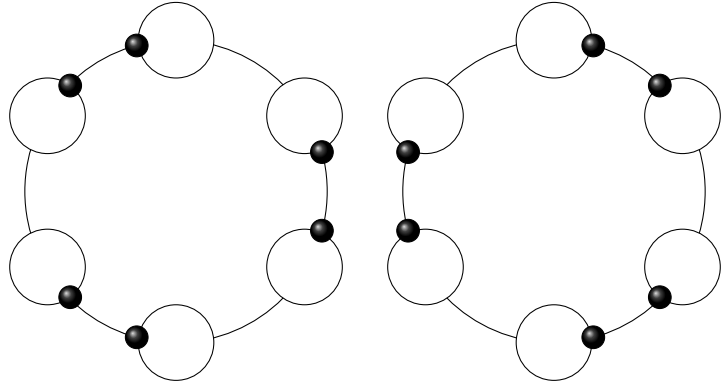
\end{example}

\subsection{Identifying Prym Divisors of High Rank}
\label{Sec:FirstContainment}

In this section, we prove half of Theorem~\ref{Thm:Generic}.  Specifically:

\begin{theorem}
\label{Thm:GenericBackwardContainment}
Let $\Gamma$ be an $r$-generic loop of loops and $\pi \colon \WG \to \Gamma$ the quotient by the antipodal involution.  Then
\[
V^r (\Gamma,\pi) \supseteq \bigcup T(Q),
\]
where the union is over all lingering words of length $g-1$ that contain a reduced subword for $w_0$.
\end{theorem}

To do this, we will require some preliminary lemmas for computing the ranks of divisors on metric graphs.

\begin{lemma}
\label{Lem:DistinctSlopes}
Let $D$ be a divisor on a metric graph $\Gamma$, and let $\varphi_0 , \ldots , \varphi_r \in R(D)$.  Let $p_1 , \ldots , p_r \in \Gamma$ be points of valence 2 such that:
\begin{enumerate}
\item  $p_i \notin \mathrm{Supp}(D)$ for all $i$,
\item  each function $\varphi_j$ is linear in a neighborhood of each point $p_i$, and
\item  for all $i$, the functions $\varphi_0 , \ldots , \varphi_r$ have distinct slopes in a neighborhood of $p_i$.
\end{enumerate}
Then there exist $a_0 , \ldots , a_r \in \R$ such that
\[
\ddiv \Big( \min_{j=0}^r \{ a_j +\varphi_j \} \Big) + D \geq p_1 + \cdots + p_r .
\]
\end{lemma}

\begin{proof}
Let $A$ be the $r \times (r+1)$ matrix with entries $A_{ij} = \varphi_j (p_i)$.  Since this matrix has tropical rank at most $r$, there exist $a_0 , \ldots , a_r \in \R$ such that $\min \{ a_j +\varphi_j \}$ occurs at least twice at each point $p_i$.  Let $\psi = \min \{ a_j +\varphi_j \}$.  Since the two functions that achieve the minimum at $p_i$ have distinct slopes, we see that $\ord_{p_i} (\psi) \geq 1$ for all $i$, and the result follows.
\end{proof}

Recall that an \emph{edge cover} of a finite graph $G$ is a collection $C$ of edges such that each vertex in $G$ is incident to at least one edge in $C$.

\begin{corollary}
\label{Cor:HowToComputeRank}
Let $D$ be a divisor on a metric graph $\Gamma$, and let $\varphi_0 , \ldots , \varphi_r \in R(D)$.  Let $G$ be a model for $\Gamma$ with no loops, and let $C$ be an edge cover for $G$ such that $\varphi_0 , \ldots , \varphi_r$ have distinct slopes along every tangent vector based at every point in the interior of an edge in $C$.  Then $D$ has rank at least $r$.
\end{corollary}

\begin{proof}
Let $p_1 , \ldots , p_r$ be distinct points, not contained in $\mathrm{Supp}(D)$, each in the interior of an edge of $C$.  By Lemma~\ref{Lem:DistinctSlopes}, there exists a function $\psi \in R(D)$ such that $\ddiv(\psi) + D \geq p_1 + \cdots + p_r$.  The set $\vert D \vert$ is compact, hence the set
\[
\{ E \in \mathrm{Eff}^r (\Gamma) \mid \ddiv (\psi) + D \geq E \text{ for some } \psi \in R(D) \}
\]
is closed.  Thus, for any effective divisor $E$ of degree $r$ supported on $C$, there exists a function $\psi \in R(D)$ with $\ddiv(\psi) +D \geq E$.  But since $C$ is an edge cover, this statement holds for all effective divisors $E$ of degree $r$ supported on the vertices of $G$.  By \cite[Theorem~1.5]{Luo11}, the vertices of $G$ are a rank-determining set, and the result follows.
\end{proof}

To show that a divisor class $[D]$ on $\WG$ has rank at least $r$, our goal is to construct a set of functions $\varphi_0 , \ldots , \varphi_r \in R(D)$ satisfying the hypotheses of Corollary~\ref{Cor:HowToComputeRank}.  These functions will also be used in Section~\ref{Sec:Lifting} to prove a lifting result for divisors in $V^r (\Gamma,\pi)$.  To construct these functions, we require some preliminary notation.  Let $Q = s_1 s_2 \cdots s_{g-1}$ be a reduced lingering word for $w_0$.  Let $Q_1 = s_1 \cdots s_{g-1} s_g \cdots s_{2g-2}$, and for any $j$ in the range $1 \leq j \leq 2g-1$, let $\sigma_j = \prod_{i=1}^{j-1} s_i$.  

For each $i \in \{ 0, \ldots , r \}$, define 
\begin{align*}
J^-(i) &= \{ j \mid \sigma_j (i) > \sigma_{j+1} (i) \} \\
J^+(i) &= \{ j \mid \sigma_j (i) < \sigma_{j+1} (i) \} \\
J(i) &= J^-(i) \cup J^+(i).
\end{align*}
Since $Q$ is reduced, $\# \{ j \in J^-(i) \mid 1 \leq j \leq g-1 \} = i$ and $\# \{ j \in J^+ (i) \mid 1 \leq j \leq g-1 \} = r-i$.  Note also that $j \in J^{\pm}(i)$ if and only if $g-1+j \in J^{\mp}(i)$, hence $\# J^- (i) = \# J^+ (i) = r$.

\begin{proposition}
\label{Prop:Fns}
Let $Q = s_1 s_2 \cdots s_{g-1}$ be a reduced lingering word for $w_0$, and let $D \in T(Q)$.  Then there exist functions $\varphi_0 , \ldots , \varphi_r \in R(D)$ such that:
\begin{enumerate}
\item  if $r$ is odd and $j \neq 1, g$, or if $r$ is even, then $\varphi_i \vert_{\beta_j}$ is linear with slope $s^r_j [\sigma_j (i)]$ in the direction from $w_{j-1}$ to $v_j$,
\item  if $r$ is odd, then $\varphi_i \vert_{\beta_1}$ has constant slope $s^r_1 [i]$ along the interval from $p_0$ to $v_1$ and constant slope $s'^r_{2g-2} [i]$ along the interval from $w_{2g-2}$ to $p_0$, and
\item if $r$ is odd, then $\varphi_i \vert_{\beta_g}$ has constant slope $s^r_g [\sigma_g (i)]$ along the interval from $\overline{p}_0$ to $v_g$ and constant slope $s'^r_{g-1} [\sigma_g (i)]$ along the interval from $w_{g-1}$ to $\overline{p}_0$.
\end{enumerate}
Moreover, setting $D_i = \ddiv(\varphi_i) +D$, we have:
\begin{enumerate}
\item  if $j \notin J(i)$, then $D_i \vert_{\gamma_j} = \langle \xi (D)_j - s^r_j [\sigma_j (i)] \rangle_j$,
\item  if $j \in J^+ (i)$, then $D_i \vert_{\gamma_j} = 0$, and
\item  if $j \in J^- (i)$, then $D_i \vert_{\gamma_j} = v_j + w_j$.
\end{enumerate}
\end{proposition}

\begin{proof}
To construct the functions $\varphi_i$, we first construct a function $\varphi_{ij}$ on each edge $\beta_j$, and a function $\varphi'_{ij}$ on each loop $\gamma_j$.  Each of these functions will be defined up to translation by a constant.  We will then show that these functions can be ``patched'' together to obtain a function $\varphi_i \in R(D)$.

The functions $\varphi_{ij}$ are described in the statement of the proposition.  Specifically, if $r$ is odd and $j \neq 1, g$, or if $r$ is even, then define $\varphi_{ij}$ to be a linear function on $\beta_j$ with slope $s^r_j [\sigma_j(i)]$ in the direction from $w_{j-1}$ to $v_j$.  If $r$ is odd, then define $\varphi_{i1}$ to be a function with constant slope $s^r_1 [i]$ along the interval from $p_0$ to $v_1$, and with constant slope $s'^r_{2g-2} [i]$ along the interval from $w_{2g-2}$ to $p_0$.  Similarly, if $r$ is odd, then define $\varphi_{ig}$ to be a function with constant slope $s^r_g [\sigma_g (i)]$ along the interval from $\overline{p}_0$ to $v_g$, and with constant slope $s'^r_{g-1} [\sigma_g (i)]$ along the interval from $w_{g-1}$ to $\overline{p}_0$.

We use Example~\ref{Ex:Loop} to define the functions $\varphi'_{ij}$.  If $s_j \notin \{ \tau_{\sigma_j (i)-1}, \tau_{\sigma_j (i)} \}$, then by Example~\ref{Ex:Loop}, we see that up to translation by a constant, there is a unique function 
\[
\varphi'_{ij} \in R \Big( D_{\vert \gamma_j} + s^r_j [\sigma_j(i)] v_j -  s'^r_j [\sigma_{j+1}(i)]w_j \Big) .
\]
Similarly, if $s_j = \tau_{\sigma_j (i)}$, then by Example~\ref{Ex:Loop}, we have
\[
D_{\vert \gamma_j} + s^r_j [\sigma_j(i)] v_j \sim s'^r_j [\sigma_{j+1}(i)]w_j .
\]
Again, up to translation by a constant, there exists a unique function $\varphi'_{ij}$ such that
\[
\ddiv (\varphi'_{ij}) = s'^r_j [\sigma_{j+1}(i)]w_j - s^r_j [\sigma_j(i)] v_j - D_{\vert \gamma_j} .
\]
If $s_j =\tau_{\sigma_j (i)-1}$, then by Example~\ref{Ex:Loop}, we have
\[
D_{\vert \gamma_j} + s^r_j [\sigma_j(i)] v_j - s'^r_j [\sigma_{j+1}(i)]w_j \sim v_j + w_j.
\]
Up to translation by a constant, there exists a unique function $\varphi'_{ij}$ such that
\[
\ddiv (\varphi'_{ij}) = (s'^r_j [\sigma_{j+1}(i)] +1)w_j - (s^r_j [\sigma_j(i)] +1) v_j - D_{\vert \gamma_j} .
\]

Now, choose coefficients $a_{ij}, a'_{ij}$ such that $a_{ij} + \varphi_{ij} (v_j) = a'_{ij} + \varphi'_{ij} (v_j)$ for all $j$, and $a_{ij} + \varphi_{ij} (w_{j-1}) = a'_{i,j-1} + \varphi'_{i,j-1} (w_{j-1})$ for all $j \neq 1$.  We define $\varphi_i$ to be a function whose restriction to $\beta_j \smallsetminus \{ w_{j-1} \}$ is $a_{ij} + \varphi_{ij}$ and whose restriction to $\gamma_j \smallsetminus \{ v_j \}$ is $a'_{ij} + \varphi'_{ij}$ for all $j$.  Our choice of coefficients guarantees that the function $\varphi_i$ is continuous at all points of $\WG$ except possibly $w_{2g-2}$.  To see that $\varphi_i$ is also continuous at $w_{2g-2}$, let $L$ denote the limit of $\varphi_i$ as you approach $w_{2g-2}$ along the tangent direction in $\beta_1$.  We must show that $\varphi_i (w_{2g-2}) = L$.

Since $Q$ is a reduced lingering word for $w_0$, by Lemma~\ref{Lem:Containsw0}, the ordered product of the terms in any half word $Q_1^j$ is equal to $w_0$.  It follows that $s^r_{g-1+j} [\sigma_{g-1+j} (i)] = - s^r_j [\sigma_j(i)]$ for all $j$.  In other words, the slope of $\varphi_{ij}$ along any tangent vector in $\beta_j$ is the negative of the slope of $\varphi_{i,g-1+j}$ along the antipodal tangent vector.  Thus,
\[
\varphi_{ij} (v_j) - \varphi_{ij} (w_{j-1}) = \varphi_{i,g-1+j} (v_{g-1+j}) - \varphi_{i,g-1+j} (w_{g-2+j}) .
\]

Similarly, because $\xi_{g-1+j} (D) = -1 - \xi_j (D)$ and $s_{g-1+j} = \overline{s}_j$, we see that the slope of $\varphi'_{ij}$ at every point $\langle \xi \rangle_j \in \gamma_j$ in the clockwise direction is equal to the slope of $\varphi'_{i,g-1+j}$ at $\langle 1 - \xi \rangle_{g-1+j} \in \gamma_{g-1+j}$ in the counterclockwise direction.  Thus,
\[
\varphi'_{ij} (w_j) - \varphi'_{ij} (v_j) = \varphi'_{i,g-1+j} (w_{g-1+j}) - \varphi'_{i,g-1+j} (v_{g-1+j}) .
\]
Putting this together, we see that
\[
\varphi_i (w_{2g-2}) - L = \sum_{j=1}^{2g-2} \Big( (\varphi_{ij} (v_j) - \varphi_{ij} (w_{j-1})) + (\varphi'_{ij} (w_j) - \varphi'_{ij} (v_j)) \Big) = 0,
\]
hence $\varphi_i$ is continuous at $w_{2g-2}$.
 
We now show that $\varphi_i \in R(D)$.  By construction, we have $\ord_p (\varphi_{ij}) \geq -D(p)$ for all points $p$ in the interior of $\beta_j$ and $\ord_p (\varphi'_{ij}) \geq -D(p)$ for all points $p$ in the interior of $\gamma_j$.  Also by construction,
\begin{align*}
\ord_{v_j} (\varphi_i) &= \ord_{v_j} (\varphi_{ij}) + \ord_{v_j} (\varphi'_{ij}) = s^r_j [\sigma_j(i)]+ \ord_{v_j} (\varphi'_{ij}) \geq -D(v_j) \\
\ord_{w_j} (\varphi_i) &= \ord_{w_j} (\varphi_{i,j+1}) + \ord_{w_j} (\varphi'_{ij}) = s'^r_j [\sigma_{j+1} (i)] + \ord_{w_j} (\varphi'_{ij}) \geq -D(w_j).
\end{align*}
We therefore have $\ord_p (\varphi_i) \geq D(p)$ for all $p \in \WG$.
\end{proof}

\begin{proof}[Proof of Theorem~\ref{Thm:GenericBackwardContainment}]
Let $Q = s_1 s_2 \cdots s_{g-1}$ be a lingering word that contains a reduced subword for $w_0$, and let $[D] \in T(Q)$.  By Lemma~\ref{Lem:Contains}, we may assume that $Q$ is a reduced lingering word for $w_0$.  It suffices to check that the functions $\varphi_0 , \ldots , \varphi_r$ defined in Proposition~\ref{Prop:Fns} satisfy the hypotheses of Corollary~\ref{Cor:HowToComputeRank}.  This is because the edges $\beta_1 , \ldots , \beta_{2g-2}$ form an edge cover of a model for $\Gamma$ with no loops.  Because $\sigma_j$ is a permutation for all $j$, the values $\sigma_j (0) , \ldots , \sigma_j (r)$ are distinct, hence so are the slopes $s^r_j [\sigma_j(i)]$.  Similarly, the slopes $s'^r_j [\sigma_j(i)]$ are distinct.  Thus, the functions $\varphi_0 , \ldots , \varphi_r$ have distinct slopes along every tangent vector in $\beta_j$ for all $j$, and the result follows.
\end{proof}

\subsection{There are No Other Prym Divisors of High Rank}
\label{Sec:SecondContainment}

In this section, we complete the proof of Theorem~\ref{Thm:Generic}.  This will require some technical arguments about Prym words in the Coxeter group $A_r$.  Recall that we write $w_1$ for the maximal element, with respect to Bruhat order, in the subgroup $S_r \subset S_{r+1}$ of permutations that fix the element $r$.  We make heavy use of the following definition.

\begin{definition}
\label{Def:FullRank}
Let $r \geq 2$.  A Prym word $Q$ in the Coxeter group $A_r$ has \emph{full rank} if every half word $Q^j$ contains a reduced subword for $w_1$.
\end{definition}

We will show in Proposition~\ref{Prop:FullRank} below that, if $Q$ has full rank, then every half word $Q^j$ contains a reduced subword for $w_0$.  We first need some preliminary lemmas.

\begin{lemma}
\label{Lem:TableauCompare}
Let $Q$ be a Prym word of full rank and let $\sigma \in S_{r+1}$ be a permutation such that $\sigma^{-1} (0) < \sigma^{-1} (r)$.  Then very half word $Q^j$ contains a reduced subword for $\sigma$.
\end{lemma}

\begin{proof}
Let $Q$ be a Prym word of full rank.  By definition, every half word $Q^j$ contains a reduced subword for $w_1$.  By definition of a Prym word, if $Q^j$ contains a reduced subword for $w_1$, then $Q^{g-1+j}$ contains a reduced subword for $\overline{w}_1$.  Since $j$ is arbitrary, every half word $Q^j$ contains a reduced subword for $\overline{w}_1$.  Since the Demazure product $\delta (Q^j)$ is the unique maximal element of $S_{r+1}$ for which $Q^j$ contains a reduced subword, we see that $\delta (Q^j) [x,y] \geq \max \{ w_1 [x,y], \overline{w}_1 [x,y] \}$ for all $x,y$.  It follows that, if $\sigma [x,y] \leq \max \{ w_1 [x,y], \overline{w}_1 [x,y] \}$ for all $x$ and $y$, then $Q^j$ contains a reduced subword for $\sigma$.  But
\[
\max \{ w_1 [x,y] , \overline{w}_1 [x,y] \} = \left\{ \begin{array}{ll}
w_1 [x,y] = x+1 &\text{if } x+y < r \\
\overline{w}_1 [x,y] = r-y+1 &\text{if } x+y > r \\
w_1 [x,y] = \overline{w}_1 [x,y] = x &\text{if } x+y = r.
\end{array} \right.
\]
Note that every permutation $\sigma$ trivially satisfies $\sigma [x,y] \leq \min \{ x+1, r-y+1 \}$.
Thus, the inequality $\sigma [x,y] > \max \{ w_1 [x,y], \overline{w}_1 [x,y] \}$ can only hold if $x+y=r$, in which case it reduces to $\sigma [x,r-x] > x$.

Now, if $\sigma^{-1} (0) < \sigma^{-1} (r)$, then for any $x < r$, the set $\{ i \mid \sigma (i) \geq r-x \}$ either contains $r$ or does not contain 0.  Hence
\[
\sigma [x,r-x] = \# \{ i \leq x \mid \sigma (i) \geq r-x \} \leq x .
\]
\end{proof}

\begin{lemma}
\label{Lem:FullRankNu}
If $Q$ is a Prym word of full rank, then very half word contains a reduced subword for $\nu_i$, where
\[
\nu_i (j) := \left\{ \begin{array}{ll}
j+i &\text{if } j \leq r-i \\
j+i-r-1 &\text{if } j \geq r-i+1.
\end{array} \right.
\]
\end{lemma}

\begin{proof}
Let $Q$ be a Prym word of full rank.  In one-line notation, $\nu_i$ is the permutation
\[
r-i+1, r-i+2, \ldots , r-1, r, 0, 1, \ldots , r-i
\]
and $\nu_i \tau_{i-1}$ is the permutation
\[
r-i+1, r-i+2, \ldots , r-1, 0, r, 1, \ldots , r-i.
\]
Since 0 preceeds $r$ in this expression, by Lemma~\ref{Lem:TableauCompare}, $Q^j$ contains a reduced subword for the permutation $\nu_i \tau_{i-1}$.  The only consecutive symbols that are inverted by the permutation $\nu_i \tau_{i-1}$ are $r-i$ and $r-i+1$, so the first term of any reduced subword for $\nu_i \tau_{i-1}$ must be $\tau_{r-i}$.  Let $s_{j'}$ be the last term in $Q^j$ that is equal to $\tau_{r-i}$.  Then the half word $Q^{j'+1}$ contains a reduced subword for $\nu_i \tau_{i-1}$.  By construction, this reduced subword must be contained in $s_{g-1+j} s_{g+j} \cdots s_{g-1+j'}$.  Since $Q$ is a Prym word, it follows that the word $s_j s_{j+1} \cdots s_{j'}$ contains a reduced subword for $\overline{\nu}_i \overline{\tau}_{i-1} = \nu_{r-i-1} \tau_{r-i}$.  If the last term of this reduced subword is $s_{j'}$, then by removing it, we obtain a subword of $Q^j$ for $\nu_{r-i-1}$.  Similarly, if the last term of this reduced subword is not $s_{j'}$, then by appending it we obtain a subword of $Q^j$ for $\nu_{r-i-1}$.  Finally, since $Q$ is a Prym word, it follows that the half word $Q^{g-1+j}$ contains a reduced subword for $\overline{\nu}_{r-i-1} = \nu_i$.  Since $j$ is arbitrary, we see that every half word $Q^j$ contains a reduced subword for $\nu_i$.
\end{proof}

\begin{proposition}
\label{Prop:FullRank}
If $Q$ is a Prym word of full rank, then every half word $Q^j$ contains a reduced subword for $w_0$.
\end{proposition}

\begin{proof}
By Lemma~\ref{Lem:FullRankNu}, the half word $Q^j$ contains a reduced subword for $\nu_i$.  Since the Demazure product $\delta (Q^j)$ is the unique maximal element of $S_{r+1}$ for which $Q^j$ contains a reduced subword, we see that $\delta (Q^j) [r-i,i] \geq \nu_i [r-i,i] = r+1-i$ for all $i$.  But this implies that $\delta (Q^j) = w_0$.
\end{proof}

We now complete the proof of Theorem~\ref{Thm:Generic}.

\begin{proof}[Proof of Theorem~\ref{Thm:Generic}]
Let $[D] \in V^r (\Gamma,\pi)$.  Since $[D] \in T(Q^1_D)$, it suffices to show that every half word $Q^j_D$ contains a reduced subword for $w_0$.  We prove this by induction on $r$.  For the base case $r=1$, since $\rk (D) \geq 1$, we have $\rk (D) \geq 0$.  By Lemma~\ref{Lem:BaseCase}, we have $\xi_j (D) = -1$ for some $j \in \{ 1, 2, \ldots , g-1 \}$.  In other words, the half word $Q^1_D$ contains the subword $s_0$, which is a reduced subword for $w_0$.  By Lemma~\ref{Lem:Containsw0}, every half word $Q^j_D$ contains a reduced subword for $w_0$.

For the inductive step, let $p \in \beta_{\ell}$.  Since $\rk (D) \geq r$, we have $\rk (D-p) \geq r-1$.  Setting $D' = D-p+\overline{p}$, we have $\rk (D') \geq r-1$ as well.  Let $Q^{\ell}_D = s_1 s_2 \cdots s_{g-1}$ and $Q^{\ell}_{D'} = s'_1 s'_2 \cdots s'_{g-1}$.  Our goal is to compute $s'_j$ in terms of $s_j$.  To do this, we must compute the divisor equivalent to $D'$ in $\widehat{\mathbb{T}}^{r-1}$.  By Lemma~\ref{Lem:Bijection}, there exists a function $\varphi \in R(D')$, unique up to translation,  such that $D' + \ddiv (\varphi) \in \widehat{\mathbb{T}}^{r-1}$.  There is a unique path in $\WG$ from $\overline{p}$ to $p_0$ that does not pass through $\{ \overline{p}_0 \} \cup \mathrm{Supp} (D')$, and similarly a unique path from $\overline{p}_0$ to $p$ that does not pass through $\{ p_0 \} \cup \mathrm{Supp} (D')$.  The function $\varphi$ has constant slope 1 along the edges $\beta_j$ in both of these paths.  From this it follows that, if $\ell \in \{ 1, \ldots , g-1 \}$, then
\[
\xi_j (D') = \left\{ \begin{array}{ll}
\xi_j (D) + 1 &\text{if } j \in \{ \ell , \ldots , g-2 \} \\
\xi_j (D) &\text{if } j \in \{ g-1, \ldots , g+ \ell \} ,
\end{array} \right.
\]
and if $\ell \in \{ g, \ldots , 2g-2 \}$, then
\[
\xi_j (D') = \left\{ \begin{array}{ll}
\xi_j (D) &\text{if } j \in \{ \ell , \ldots , 2g-2 \} \\
\xi_j (D) + 1 &\text{if } j \in \{ 1, \ldots , g+ \ell \} .
\end{array} \right.
\]
By the formula for $s^r_j [i]$, we then have
\[
s'_j = \left\{ \begin{array}{ll}
s_j &\text{if } s_j \neq \tau_{r-1} \\
\epsilon &\text{if } s_j = \tau_{r-1}.
\end{array} \right.
\]

By induction, the half word $Q^{\ell}_{D'}$ contains a reduced subword for $w_0$, hence the half word $Q^{\ell}_D$ contains a reduced word for $w_1$.  Since $\ell$ is arbitrary, $Q_D$ has full rank, and the result follows from Proposition~\ref{Prop:FullRank}.
\end{proof}

Before moving on, we note the following similarity with classical algebraic geometry.

\begin{corollary}
\label{Cor:RankParity}
If $[D] \in P^r$, then $\rk (D) \equiv r \Mod{2}$.
\end{corollary}

\begin{proof}
Let $[D] \in P^r$ and suppose that $\rk (D) \geq r-1$.  We prove, by induction on $r$, that $\rk (D) \geq r$.  For the base case $r=1$, if $[D] \in P^1$ and $\rk (D) \geq 0$, then by Lemma~\ref{Lem:BaseCase}, every half word of $Q_D$ contains $w_0$.  Thus, $[D] \in T(Q_D)$, which is contained in $V^1 (\Gamma,\pi)$ by Theorem~\ref{Thm:Generic}, hence $\rk (D) \geq 1$.

For the inductive step, let $p \in \beta_{\ell}$ and let $D' = D-p+\overline{p}$.  Then $\rk(D') \geq r-2$, hence by induction, we have $\rk (D') \geq r-1$.  By the proof of Theorem~\ref{Thm:Generic}, we see that the half word $Q^{\ell}_D$ contains a reduced word for $w_1$.  Since $\ell$ is arbitrary, $Q_D$ has full rank, and the result follows from Proposition~\ref{Prop:FullRank}.
\end{proof}

\subsection{Topological Properties}
\label{Sec:Tropology}

In this section, we discuss several topological properties of $V^r (\Gamma,\pi)$.  Everything in this section can be deduced from combinatorial properties of a certain poset, which can be defined for an arbitrary Coxeter group.

\begin{definition}
\label{Def:Poset}
Let $G=(W,S)$ be a Coxeter group, $n$ a positive integer, and $w \in W$.  The \emph{lingering subword poset} $P(G,w,n)$ is the set of lingering words in $G$ of length $n$ that contain a reduced subword for $w$, with partial ordering given by 
\[
Q \leq Q' \text{ if } Q' \text{ is a lingering subword of } Q.
\]
\end{definition}

The lingering subword poset is graded, where the rank of a lingering word $Q = s_1 \cdots s_n$ is given by $\rho (Q) := \# \{ j \mid s_j = \epsilon \}$.  For a fixed lingering word $Q$, the principal upper ideal
\[
\{ Q' \in P(G,w,n) \mid Q' \text{ is a lingering subword of } Q \}
\]
is the face poset of a simplicial complex, known as the \emph{subword complex} $\Delta (Q,w)$.  It is shown in \cite{KnutsonMiller} that $\Delta (Q,w)$ is homemorphic to a sphere if $\delta (Q) = w$ and homeomorphic to a ball otherwise.

If two lingering words $Q = s_1 \cdots s_n$ and $Q' = s'_1 \cdots s'_n \in P(G,w,n)$ have a lower bound, then they have a unique maximal lower bound $Q \wedge Q' = s''_1 \cdots s''_n$, given by
 \[
s''_j := \left\{ \begin{array}{ll}
s_j &\text{if } s_j \neq \epsilon \\
s'_j &\text{otherwise} .
\end{array} \right.
\].
Similarly, if $Q$ and $Q'$ have an upper bound, then they have a unique minimal upper bound $Q \vee Q' = s''_1 \cdots s''_n$, given by
 \[
s''_j := \left\{ \begin{array}{ll}
s_j &\text{if } s'_j = s_j \\
\epsilon &\text{otherwise} .
\end{array} \right.
\].
In other words, the poset obtained from $P(G,w,n)$ by adding a minimal element $\hat{0}$ and a maximal element $\hat{1}$ is a lattice.

By Theorem~\ref{Thm:Generic}, we have $V^r (\Gamma,\pi) = \cup T(Q)$, where the union is over elements of $P(A_r,w_0,g-1)$.  By Lemma~\ref{Lem:Contains}, one has $T(Q) \subseteq T(Q')$ if and only if $Q \leq Q'$.  The intersections of the tori $T(Q)$ are also encoded by the poset $P(A_r,w_0,g-1)$.

\begin{proposition}
\label{Prop:ElementaryFactsAboutTQ}
Let $Q = s_1 \cdots s_{g-1}$ and $Q' = s'_1 \cdots s'_{g-1}$ be lingering words.  Then $T(Q) \cap T(Q')$ is nonempty if and only if $Q$ and $Q'$ have a common lower bound in $P(A_r,w_0,g-1)$.  In this case, $T(Q) \cap T(Q') = T(Q \wedge Q')$.
\end{proposition}

\begin{proof}
By Lemma~\ref{Lem:Contains}, if $Q''$ is a lingering word such that $Q'' \leq Q$ and $Q'' \leq Q'$, then $T(Q'') \subseteq T(Q) \cap T(Q')$.  Conversely, if no such lingering subword $Q''$ exists, then there exists a $j$ such that $s_j, s'_j \neq \epsilon$, and $s_j \neq s'_j$.  Suppose $s_j = \tau_i$ and let $[D'] \in T(Q')$.  Since $s'_j \notin \{ \tau_i , \epsilon \}$, we have $\xi_j (D') \neq s_j^r [i]$, hence $[D'] \notin T(Q)$, and $T(Q) \cap T(Q') = \emptyset$. 

If $[D] \in T(Q) \cap T(Q')$, then $\xi_j (D) = s^r_j [i]$ whenever $s_j$ or $s'_j$ is equal to $\tau_i$, hence $[D] \in T(Q \wedge Q')$.  It follows that $T(Q) \cap T(Q') = T(Q \wedge Q')$.
\end{proof}

We list some basic properties of the lingering subword poset, and the corresponding topological properties of $V^r (\Gamma,\pi)$.

\begin{lemma}
\label{Lem:PosetIsPureDimension}
If $Q \in P(G,w,n)$ is maximal, then $\rho (Q) = n-\ell(w)$.  In particular, every maximal element of the lingering subword poset has the same rank.
\end{lemma}

\begin{proof}
If $Q$ is maximal, then $Q$ contains a reduced subword for $w$, but no proper lingering subword of $Q$ contains a reduced subword for $w$.  It follows that $Q$ is a reduced lingering word for $w$, hence $\rho (Q) = \# \{ j \mid s_j = \epsilon \} = n - \ell(w)$.
\end{proof}

\begin{corollary}
\label{Cor:PureDimension}
The tropical Prym-Brill-Noether variety $V^r (\Gamma,\pi)$ is pure dimensional, of dimension $g-1-{{r+1}\choose{2}}$.
\end{corollary}

\begin{proof}
By Lemma~\ref{Lem:Contains}, we have $T(Q) \subseteq T(Q')$ if and only if $Q \leq Q'$ in $P(A_r,w_0,g-1)$.  Moreover, the dimension of $T(Q)$ is $\rho(Q)$.  It follows from Lemma~\ref{Lem:PosetIsPureDimension} that, if $T(Q)$ is maximal with respect to containment, then
\[
\mathrm{dim} T(Q) = \rho(Q) = g-1-{{r+1}\choose{2}}.
\]
\end{proof}

By \cite[Lemma~3.5]{KnutsonMiller}, if $\rho (Q) = n - \ell (w) -1$, then there are at most two elements of $P(G,w,n)$ that are strictly greater than $Q$.  Moreover, there are exactly two such elements if and only if $\delta (Q) = w$.  It follows that there is a graph, known as the \emph{flip graph} $\cF (G,w,n)$ whose vertices are the elements of $P(G,w,n)$ of maximal rank $n-\ell (w)$, and whose edges are the elements $Q \in P(G,w,n)$ of rank $n-\ell (w)-1$ with $\delta (Q) = w$. 

\begin{lemma}
\label{Lem:BraidRelations}
If $n > \ell(w)$, then the graph $\cF(G,w,n)$ is connected.
\end{lemma}

\begin{proof}
If $Q$ and $Q'$ are maximal, then they are reduced lingering words for $w$.  By Matsumoto's theorem \cite[Theorem~3.3.1]{BjornerBrenti}, any reduced word for $w$ can be obtained from any other by a sequence of braid relations.  It follows that $Q'$ can be obtained from $Q$ by applying the following operations:
\begin{enumerate}
\item  shift operation:  replace $s \epsilon$ with $s s$, and then with $\epsilon s$, and
\item  braid operation:  replace $s_1 s_2 s_1 \cdots s_2 \epsilon$ with $s_1 s_2 s_1 \cdots s_2 s_1$, and then with $\epsilon s_2 s_1 \cdots s_2 s_1$.
\end{enumerate}
Since both operations first replace an instance of $\epsilon$ with a non-identity element, and then replace a non-identity element with $\epsilon$, we see that both are cover relations in the lingering subword poset $P(G,w,n)$, and the result follows.
\end{proof}

\begin{corollary}
\label{Cor:Connected}
If $g-1 >{{r+1}\choose{2}}$, then $V^r (\Gamma,\pi)$ is connected in codimension 1.
\end{corollary}

\begin{proof}
If $T(Q)$ and $T(Q')$ are maximal, then by Lemma~\ref{Lem:BraidRelations}, there exists a sequence
\[
Q = Q_0 , Q_1 , \ldots , Q_m = Q'
\]
where $T(Q_{i-1}) \cap T(Q_i) = T(Q_{i-1} \wedge Q_i)$ has codimension 1 in both $T(Q_{i-1})$ and $T(Q_i)$.
\end{proof}

Although we will not use it, we note the following property of the flip graph.

\begin{lemma}
\label{Lem:Bipartite}
The flip graph $\cF (G,w,n)$ is bipartite if and only if $n = \ell (w) +1$.
\end{lemma}

\begin{proof}
First, suppose that $n > \ell (w) +1$, and let $Q = s_1 \cdots s_{n-2}$ be a reduced word for $w$.  Then the three vertices of $\cF (G,w,n)$ corresponding to $\epsilon \epsilon s_1 s_2 \cdots s_{n-2}$, $\epsilon s_1 \epsilon s_2 \cdots s_{n-2}$, and $s_1 \epsilon \epsilon s_2 \cdots s_{n-2}$ form a triangle, hence the graph is not bipartite.

Now, suppose that $n = \ell (w) + 1$.  In \cite{BCL}, the authors define a sign function on reduced words for $w$.  We use this to define a sign function on the vertices of $\cF (G,w,n)$.  Specifically, let $Q = s_1 \cdots s_n$ be a reduced lingering word for $w$ of length $\ell (w) + 1$.  Then exactly one term of $Q$ is equal to $\epsilon$, say $s_i = \epsilon$.  We define the sign of $Q$ to be $(-1)^i \mathrm{sign} (Q \smallsetminus s_i )$.  Now, suppose that $Q = s_1 \cdots s_n$ and $Q' = s'_1 \cdots s'_n$ are adjacent in the flip graph and $s_i =s'_j = \epsilon$.  We show that $Q$ and $Q'$ have opposite sign.  By \cite[Lemma~3.6]{BCL}, we have
\begin{align*}
\mathrm{sign}(Q') &= (-1)^j \mathrm{sign} (Q' \smallsetminus s'_j) = (-1)^{i-1} \mathrm{sign} (Q \smallsetminus s_i) = - \mathrm{sign} (Q) .
\end{align*}
\end{proof}

When $n = \ell (w)$, the lingering subword poset $P(G,w,n)$ is simply the set of reduced words for $w$.

\begin{lemma}
\label{Lem:Cardinality}
If $g-1= {{r+1}\choose{2}}$, then $\# V^r (\Gamma,\pi) = 2^{{r}\choose{2}} \cdot (g-1)! \cdot \prod_{i=1}^r \frac{(i-1)!}{(2i-1)!}$.
\end{lemma}

\begin{proof}
If $g-1= {{r+1}\choose{2}}$, then $\# V^r (\Gamma,\pi)$ is equal to the number of reduced words for $w_0$.  By \cite{Stanley84, EdelmanGreene87}, this is equal to the number of standard Young tableaux on the isosceles triangular partition with side length $r$.  This number can be computed explicitly via the hook-length formula.
\end{proof}

When the dimension of $V^r (\Gamma,\pi)$ is positive, we generalize Lemma~\ref{Lem:Cardinality} as follows.

\begin{lemma}
\label{Lem:Count}
If $g-1\geq {{r+1}\choose{2}}$, then there exist divisors $F_1 , \ldots , F_{g-1-{{r+1}\choose{2}}} \in P^r - P^1$ such that
\[
X = \bigcap_{i=1}^{g-1-{{r+1}\choose{2}}} [V^1 (\Gamma,\pi) + F_i ]
\]
is finite and
\[
\# X = 2^{{r}\choose{2}} \cdot (g-1)! \cdot \prod_{i=1}^r \frac{(i-1)!}{(2i-1)!}.
\]
Moreover, if $Q^1_D$ is a reduced lingering word, then the divisors $F_i$ can be chosen so that $[D] \in X$.
\end{lemma}

\begin{proof}
By \cite{Stanley84, EdelmanGreene87}, the number of reduced words for $w_0$ is
\[
2^{{r}\choose{2}} \cdot {{r+1}\choose{2}}! \cdot \prod_{i=1}^r \frac{(i-1)!}{(2i-1)!}.
\]
Given a lingering word for $w_0$, one obtains a lingering reduced word for $w_0$ of length $g-1$ by arbitrarily inserting $g-1-{{r+1}\choose{2}}$ copies of $\epsilon$.  It follows that the number of lingering reduced words for $w_0$ of length $g-1$ is
\[
2^{{r}\choose{2}} \cdot {{g-1}\choose{{{r+1}\choose{2}}}} \cdot {{r+1}\choose{2}}! \cdot \prod_{i=1}^r \frac{(i-1)!}{(2i-1)!}.
\]
By Theorem~\ref{Thm:Generic}, $V^r (\Gamma,\pi)$ is a union of precisely this number of translates of coordinate tori.  By the same reasoning, $V^1 (\Gamma,\pi)$ is a union of translates of the $g-1$ codimension-1 coordinate tori.  If $Q$ is a lingering reduced word for $w_0$ of length $g-1$, then $T(Q)$ intersects $g-1-{{r+1}\choose{2}}$ general translates of $V^1 (\Gamma,\pi)$ in $\Big(g-1-{{r+1}\choose{2}}\Big)!$ points.  It follows that the intersection of $V^r (\Gamma,\pi)$ with such translates is a union of
\begin{align*}
2^{{r}\choose{2}} \cdot {{g-1}\choose{{{r+1}\choose{2}}}} \cdot {{r+1}\choose{2}}! \cdot \Big(g-1-{{r+1}\choose{2}}\Big)! \cdot \prod_{i=1}^r \frac{(i-1)!}{(2i-1)!} \\
= 2^{{r}\choose{2}} \cdot (g-1)! \cdot \prod_{i=1}^r \frac{(i-1)!}{(2i-1)!}
\end{align*}
distinct points.  Furthermore, these translates can be chosen to contain a particular point $[D] \in V^r (\Gamma,\pi)$ if $V^r (\Gamma,\pi)$ is locally isomorphic to a coordinate torus in a neighborhood of $[D]$.  By Theorem~\ref{Thm:Generic}, this holds if and only if $Q^1_D$ is reduced.
\end{proof}

\section{Lifting the Prym-Brill-Noether Variety for $r$-Generic Loops of Loops}
\label{Sec:Lifting}

\subsection{Local Properties of the Tropical Prym-Brill-Noether Variety}
\label{Sec:LocalTrop}

In this section, we prove Theorem~\ref{Thm:Lifting}.  Our proof follows the same rough outline as that of the lifting result in \cite{CJP}.  In this setting, the Prym-Brill-Noether variety $V^1 (\Gamma,\pi)$ plays the role of the theta divisor on a tropical Jacobian, and we study it first.

\begin{lemma}
\label{Lem:Theta}
Let $\Gamma$ be a loop of loops (with arbitrary edge lengths), and let $\pi \colon \WG \to \Gamma$ be the antipodal involution.  Every facet of $V^1 (\Gamma,\pi)$ has multiplicity 1.
\end{lemma}

\begin{proof}
Note that every loop of loops is 1-generic, regardless of edge lengths.  Thus, by Theorem~\ref{Thm:Generic}, $V^1 (\Gamma,\pi)$ consists of translates of the $g-1$ coordinate codimension 1 subtori in $P^1$, and each of these carries a positive integer multiplicity.  By Corollary~\ref{Cor:RankParity}, a divisor class $[D] \in P^1$ has rank at least 1 if and only if it is equivalent to an effective divisor.  Similarly, if $f \colon \WC \to C$ is an \'{e}tale double cover of curves specializing to $\pi$, then a divisor class of odd rank $[D_C] \in P (C,f)$ has rank at least 1 if and only if it is equivalent to an effective divisor.  Since every effective divisor on $\WG$ lifts to an effective divisor on $\WC$, we see that
\[
\Trop (V^1 (C,f)) = V^1 (\Gamma,\pi).
\]

By Lemma~\ref{Lem:Count}, the intersection of $g-1$ general translates of $V^1 (\Gamma,\pi)$ consists of $(g-1)!$ distinct points, each with a tropical multiplicity $m$ equal to the product of the multiplicities of the facets.  By Rabinoff's lifting theorem \cite{Rabinoff12},  the intersection of $g-1$ general translates of $V^1 (C,f)$ consists of $m \cdot (g-1)!$ distinct points.  By the formula for the class of $V^r (C,f)$ from \cite{DCP}, however, the numerical class of $V^1 (C,f)$ is equal to that of a principal polarization $\xi$ for $P^1 (C,f)$.  We have the intersection number $\xi^{g-1} = (g-1)!$, so it follows that $m=1$, hence every facet of $V^1 (\Gamma,\pi)$ has multiplicity 1.
\end{proof}

\begin{definition}
We say that a divisor class $[D] \in V^r (\Gamma,\pi)$ is \emph{vertex avoiding} if the lingering word $Q^1_D$ is reduced and $\xi (D)_j \notin \{ s^r_j [0] -1, s^r_j [r] \}$ for all $j$.
\end{definition}

Note that the set of vertex avoiding classes is dense in $V^r (\Gamma,\pi)$.  The terminology is chosen to mirror that from \cite[Definition~2.3]{CJP}, where a dense subset of the Brill-Noether variety on a chain of loops is defined to be vertex avoiding.  Let $\varphi_0 , \ldots , \varphi_r \in R(D)$ be the functions defined in Proposition~\ref{Prop:Fns} and $D_i = \ddiv(\varphi_i) + D$.  By Proposition~\ref{Prop:Fns}, if $[D]$ is vertex avoiding, then the vertices $v_j$ and $w_j$ are contained in $\mathrm{Supp} (D_i)$ if and only if $j \in J^-(i)$.

The next proposition shows that, if $[D] \in V^r (\Gamma,\pi)$ is vertex avoiding, then in a neighborhood of $[D]$, the tropical Prym-Brill-Noether variety $V^r (\Gamma,\pi)$ is a local complete intersection of translates of $V^1 (\Gamma,\pi)$.

\begin{proposition}
\label{Prop:TropicalLocalCI}
Let $[D] \in V^r (\Gamma,\pi)$ be vertex avoiding.  Then there exists an open neighborhood $U$ of $[D]$ and effective divisors $E_1 , \ldots , E_{{r+1}\choose{2}}$ of degree $r-1$ such that
\[
V^r (\Gamma,\pi) \cap U = \bigcap_{j=1}^{{r+1}\choose{2}} [V^1 (\Gamma,\pi) + E_j - \overline{E}_j ] \cap U .
\]
\end{proposition}

\begin{proof}
Let $Q = s_1 s_2 \cdots s_{g-1}$.  For each $j$ such that $s_j \neq \epsilon$, there exists a unique element $i_j$ such that $j \in J^-(i_j)$.  Let
\[
E_j = \sum_{j' \in J^- (i_j)} \Big( \sum_{j' \in \{ j+1, \ldots , g-2+j \}} w_{j'} +  \sum_{j' \in \{ g+j, \ldots , j-1 \} } v_{j'} \Big).
\]
In other words, $E_j$ consists of one vertex from each loop $\gamma_{j'}$ for $j' \in J^- (i_j) \smallsetminus \{ j \}$.  If $\gamma_{j'}$ is in the connected component of $\WG \smallsetminus ( \gamma_j \cup \gamma_{g-1+j} )$ that does not contain $p_0$, then the vertex of $\gamma_{j'}$ in the support of $E_j$ is equal to $w_{j'}$, and if $\gamma_{j'}$ is in the other connected component, the vertex is equal to $v_{j'}$.  For a sufficiently small neighborhood $U$ of $[D]$, we show that
\[
V^r (\Gamma,\pi) \cap U = \bigcap_{j \in [g-1], s_j \neq \epsilon} [V^1 (\Gamma,\pi) + E_j - \overline{E}_j ] \cap U .
\]

By the definition of rank, the left-hand side is contained in the right-hand side.  To see the reverse containment, let $\varphi_{i_j} \in R(D)$ be the function constructed in Proposition~\ref{Prop:Fns}, and let $D_j = \ddiv(\varphi_{i_j}) + D - E_j + \overline{E}_j$.  By Proposition~\ref{Prop:Fns}, $D_j$ is effective, and
\begin{align*}
&D_j \vert \gamma_{g-1+j} = 0 \\
&D_j \vert \gamma_{j} = v_j + w_j \\
&\deg (D_j \vert \gamma_{j'}) = 1 \text{ for all } j' \notin \{ j, g-1+j \}.
\end{align*}
Moreover, if $j' \in \{ 1, \ldots , j-1 \}$, then $v_{j'} \notin \mathrm{Supp} (D_j)$, and if $j' \in \{ j+1, \ldots, g-1 \}$, then $w_{j'} \notin \mathrm{Supp} (D_j)$.

By Lemma~\ref{Lem:Bijection}, there exists a function $\varphi \in R(D_j)$, unique up to translation,  such that $D_j + \ddiv (\varphi) \in \widehat{\mathbb{T}}^1$.  There is a unique path in $\WG$ from $p_0$ to $v_{g-1+j}$ that does not pass through $\{ \overline{p}_0 \} \cup \mathrm{Supp} (D_j)$, and similarly a unique path from $w_j$ to $\overline{p}_0$ that does not pass through $\{ p_0 \} \cup \mathrm{Supp} (D_j)$.  The function $\varphi$ has constant slope 1 along each of the edges $\beta_{\ell}$ in both of these paths.  From the above description, we see that $\mathrm{Supp}(D_j + \ddiv (\varphi))$ contains $v_j$ but none of the vertices $v_{j'}$ for $j' \in [g-1] \smallsetminus \{ j \}$.  It follows that, if $Q^1_{D_j} = s'_1 \cdots s'_{g-1}$, then $s'_{j'} = \epsilon$ for all $j' \neq j$, and $s'_j = \tau_0$.

Now, let $[D'] \in U$.  For $U$ sufficiently small, the lingering word $Q^1_{D'} = t_1 \cdots t_{g-1}$ is a lingering subword of $Q^1_D$.  Consequently, if $D'_j = D' - E_j + \overline{E}_j$, then $Q^1_{D'_j}$ is a lingering subword of $Q^1_{D_j}$.  If $[D'] \in V^1 (\Gamma,\pi) + E_j - \overline{E}_j$, then by Lemma~\ref{Lem:BaseCase}, some term of the lingering word $Q^1_{D'_j} = t'_1 \cdots t'_{g-1}$ must be equal to $\tau_0$.  But since $Q^1_{D'_j}$ is a lingering subword of $Q^1_{D_j}$, by the above, we must have $t'_j = \tau_0$, hence $t_j = s_j$.  Thus, if $[D'] \in \cap_j [V^1 (\Gamma,\pi) + E_j - \overline{E}_j ] \cap U$, we see that $Q^1_{D'} = Q^1_D$.  By Theorem~\ref{Thm:Generic}, it follows that $[D'] \in V^r (\Gamma,\pi)$. 
\end{proof}

\subsection{Local Properties of Classical Prym-Brill-Noether Varieties}
\label{Sec:Local}

Let $f \colon \WC \to C$ be an \'{e}tale double cover of curves, let $p \in \WC$, and let $\psi \colon \Pic^{2g-2} (\WC) \to \Pic^{2g-2} (\WC)$ be the map given by $\psi (D) = D - p + \overline{p}$.  One can see from the definition of rank that there is a set-theoretic inclusion $\psi (V^r (C,f)) \subseteq V^{r-1} (C,f)$.  The next proposition shows that this an inclusion of schemes.

\begin{proposition}
\label{Prop:LocalCI}
Let $C, f$ and $\psi$ be as above.  Then there is an inclusion of schemes \[\psi (V^r (C,f)) \subseteq V^{r-1} (C,f).\]
\end{proposition}

\begin{proof}
We follow the scheme-theoretic construction of the pointed Prym-Brill-Noether variety from \cite{Tarasca}.  Let $\cL$ be a Poincar\'{e} line bundle on $\Pic^{2g-2} (\WC) \times \WC$, and let $\cE = (1 \times f)_* \cL$.  Let $\pi \colon \Pic^{2g-2} (\WC) \times C \to C$ denote the projection onto the second factor.  Fix $N$ sufficiently large, let $B$ be a sum of $N$ distinct points on $C$, none equal to $p$ or $\overline{p}$, and let
\begin{align*}
\mathcal{V} &= \pi_* (\cE (B)/\cE (-B)) \vert_{P^r (C,f)} \\
\mathcal{W}_i &= \pi_* (\cE (B-ip)) \vert_{P^r (C,f)} \\
\mathcal{U} &= \pi_* (\cE /\cE (-B)\vert_{P^r (C,f)} .
\end{align*}
Then, by \cite{Mumford71, Tarasca}, $\mathcal{V}$ is a vector bundle equipped with a nondegenerate quadratic form, $\mathcal{U}$ is a maximal isotropic subbundle, and $\mathcal{W}_r \subset \cdots \mathcal{W}_0$ is an isotropic flag of subbundles.  There is a scheme-theoretic equality
\[
V^r (C,f) = \{ L \in P^r (C,f) \mid \mathrm{rank} (\mathcal{W}_i \cap \mathcal{U})_{\vert L} \geq r+1-i \text{ for all } i \} ,
\]
where the object on the right is defined as an isotropic determinantal scheme.  From this, we see that there is an inclusion of schemes
\begin{align}
\label{eq:first}
V^r (C,f) \subseteq \{ L \in P^r (C,f) \mid \mathrm{rank} (\mathcal{W}_i \cap \mathcal{U})_{\vert L} \geq r+1-i \text{ for all } i \geq 1 \}.
\end{align}

Similarly, since $\cL' = (\psi \times \mathrm{id})^* \cL (p-\overline{p})$ is also Poincar\'{e} line bundle, we define $\cE' = (1 \times f)_* \cL'$ and
\begin{align*}
\mathcal{V}' &= \pi_* (\cE' (B)/\cE' (-B)) \vert_{P^{r-1} (C,f)} \\
\mathcal{W}'_i &= \pi_* (\cE' (B-i\overline{p})) \vert_{P^{r-1} (C,f)} \\
\mathcal{U}' &= \pi_* (\cE' /\cE' (-B)\vert_{P^{r-1} (C,f)} .
\end{align*}
And we obtain an inclusion of schemes:
\begin{align}
\label{eq:second}
V^{r-1} (C,f) \supseteq \{ L \in P^{r-1} (C,f) \mid \mathrm{rank} (\mathcal{W}'_i \cap \mathcal{U}')_{\vert L} \geq r+1-i \text{ for all } i \geq 1 \}.
\end{align}
Finally, the isomorphism $\psi$ identifies the scheme on the right in \eqref{eq:first} with that on the right in \eqref{eq:second}.
\end{proof}

\begin{corollary}
\label{Cor:LocalCI}
Let $f \colon \WC \to C$ be an \'{e}tale double cover of curves, let $E$ be an effective divisor on $\WC$ of degree $r-1$, and let $\psi \colon \Pic^{2g-2} (\WC) \to \Pic^{2g-2} (\WC)$ be the map given by $\psi (D) = D - E + \overline{E}$.  Then there is an inclusion of schemes $\psi (V^r (C,f)) \subseteq V^1 (C,f)$.
\end{corollary}

\begin{proof}
This follows by applying Proposition~\ref{Prop:LocalCI} $r-1$ times.
\end{proof}

\subsection{Proof of the Lifting Theorem}
\label{Sec:LiftingProof}

\begin{proof}[Proof of Theorem~\ref{Thm:Lifting}]
Since the set of vertex avoiding divisor classes $[D] \in V^r (\Gamma,\pi)$ is dense in $V^r (\Gamma,\pi)$ and $\Trop (V^r (C,f))$ is closed, it suffices to consider the case where $[D]$ is vertex avoiding.  By Lemma~\ref{Lem:Count}, there exists divisors $F_1 , \ldots , F_{g-1-{{r+1}\choose{2}}} \in P^r - P^1$ such that $[D]$ is contained in the intersection
\begin{align*}
X &= \bigcap_{i=1}^{g-1-{{r+1}\choose{2}}} [V^1 (\Gamma,\pi)] + F_i ], \text{ and} \\
\vert X \vert &= 2^{{r}\choose{2}} \cdot (g-1)! \cdot \prod_{i=1}^r \frac{(i-1)!}{(2i-1)!} .
\end{align*}
Let $\cF_i$ be a divisor on $C$ such that $\Trop (\cF_i) = F_i$, and define $\mathcal{X} \subseteq \Pic^{2g-2} (\WC)$ to be the intersection
\[
\mathcal{X} = \bigcap_{i=1}^{g-1-{{r+1}\choose{2}}} [V^1 (C,f) + \cF_i ] .
\]
Since $\Trop (\mathcal{X}) \subseteq X$, we see that $\mathcal{X}$ is finite.  By the formula for the class of $V^r (C,f)$ from \cite{DCP}, one sees that the length of the finite scheme $\mathcal{X}$ is exactly
\[
2^{{r}\choose{2}} \cdot (g-1)! \cdot \prod_{i=1}^r \frac{(i-1)!}{(2i-1)!},
\]
the same as the cardinality of $X$.

Now, by Proposition~\ref{Prop:TropicalLocalCI}, there exists an open neighborhood $U$ of $[D]$ and effective divisors $E_{g-{{r+1}\choose{2}}} , \ldots , E_{g-1}$ of degree $r-1$ such that
\[
V^r (\Gamma,\pi) \cap U = \bigcap_{j=g-{{r+1}\choose{2}}}^{g-1} [V^1 (\Gamma,\pi) + E_j - \overline{E}_j ] \cap U .
\]
As before, we lift each of these divisors $E_j$ to $\cE_j$ on $C$.  By Proposition~\ref{Prop:LocalCI}, $V^r (C,f)$ is contained in $V^1 (C,f) + \cE_j - \overline{\cE}_j$ for all $j$.  For $j \geq g-{{r+1}\choose{2}}$, let $F_j = E_j - \overline{E}_j$ and $\cF_j = \cE_j - \overline{\cE}_j$.

In the open neighborhood $U$, the hypersurfaces $\Trop (V^1 (C,f)) + \cF_j$ are translates of coordinate hyperplanes, all with multiplicity one by Lemma~\ref{Lem:Theta}.  It follows from \cite{Rabinoff12} that there is at most one point tropicalizing to $[D]$ in $\cap_{i=1}^{g-1} [V^1 (C,f) + \cF_i ]$.  In other words, the tropicalization map $\Trop \colon \mathcal{X} \to X$ is injective.  By the above, however, these two sets have the same cardinality, so the tropicalization map is surjective as well.
\end{proof}

\subsection{Consequences of Theorem~\ref{Thm:Lifting}}
\label{Sec:Consequences}

Theorem~\ref{Thm:Lifting} has many useful corollaries.  First, we see that, if a Prym curve specializes to $\pi \colon \WG \to \Gamma$ with $\Gamma$ $r$-generic, then the Prym-Brill-Noether variety has the expected dimension.

\begin{corollary}
\label{Cor:Dimension}
Let $\Gamma$ be an $r$-generic loop of loops and let $\pi \colon \WG \to \Gamma$ be the quotient by the antipodal involution.  If $f \colon \WC \to C$ is an \'{e}tale double cover of curves specializing to $\pi$, then
\[
\mathrm{dim} V^r (C,f) = \mathrm{dim} V^r (\Gamma,\pi) = g-1-{{r+1}\choose{2}}.
\]
\end{corollary}

\begin{proof}
By \cite[Theorem~6.9]{Gubler07}, we have $\mathrm{dim} V^r (C,f) = \mathrm{dim} \Trop (V^r (C,f))$.  By Theorem~\ref{Thm:Lifting}, we have $\Trop (V^r (C,f)) = V^r (\Gamma,\pi)$,and by Corollary~\ref{Cor:PureDimension}, we have $\mathrm{dim} V^r (\Gamma,\pi) = g-1-{{r+1}\choose{2}}$.
\end{proof}

As an immediate consequence, we obtain a new proof of Welters' theorem.

\begin{corollary}
\label{Cor:Welters}
Let $[f \colon \WC \to C] \in \mathcal{R}_g$ be a general \'{e}tale double cover of curves.  Then 
\[
\mathrm{dim} V^r (C,f) = g-1-{{r+1}\choose{2}}.
\]
\end{corollary}

\begin{proof}
Let $\Gamma$ be an $r$-generic loop of loops and let $\pi \colon \WG \to \Gamma$ be the quotient by the antipodal involution.  By \cite[Lemma~5.9]{JensenLen18}, the map $\pi \colon \WG \to \Gamma$ can be lifted to an \'{e}tale double cover $f \colon \WC \to C$.  By Corollary~\ref{Cor:Dimension}, we have 
\[
\mathrm{dim} V^r (C,f) = g-1-{{r+1}\choose{2}}.
\]
By semicontinuity, a general \'{e}tale double cover $[f \colon \WC \to C] \in \mathcal{R}_g$ satisfies $\mathrm{dim} V^r (C,f) \leq g-1-{{r+1}\choose{2}}$.

Now, let
\[
Y = \Big\{ [f \colon \WC \to C] \in \mathcal{R}_g \mid \mathrm{dim} V^r (C,f) \geq g-1-{{r+1}\choose{2}} \Big\} ,
\]
and let $Y^{\an}$ be the analytification of $Y$ in $\mathcal{R}_g^{\an}$.  By Corollary~\ref{Cor:Dimension}, the image of $Y^{\an}$ under the retraction $\mathcal{R}_g^{\an} \to \mathcal{R}_g^{\trop}$ contains the set of covers $\pi \colon \WG \to \Gamma$ where $\Gamma$ is $r$-generic, which has dimension $3g-3$.  Thus,
\[
\mathrm{dim} (Y) \geq \mathrm{dim} (\mathrm{im} (Y^{\an})) = 3g-3 = \mathrm{dim} (\mathcal{R}_g).
\]
Since $\mathcal{R}_g$ is irreducible, the result follows.
\end{proof}

In the case where $V^r (\Gamma,\pi)$ is finite, we see that it is in bijection with $V^r (C,f)$.

\begin{corollary}
\label{Cor:Bijection}
Let $\Gamma$ be an $r$-generic loop of loops, let $\pi \colon \WG \to \Gamma$ be the quotient by the antipodal involution, and let $f \colon \WC \to C$ be an \'{e}tale double cover of curves specializing to $\pi$.  If $g-1 = {{r+1}\choose{2}}$, then the tropicalization map $\Trop \colon V^r (C,f) \to V^r (\Gamma,\pi)$ is a bijection.
\end{corollary}

\begin{proof}
In the proof of Theorem~\ref{Thm:Lifting}, it is shown that the map $\Trop \colon \mathcal{X} \to X$ is a bijection.  But in the case where $g-1 = {{r+1}\choose{2}}$, we have $\mathcal{X} = V^r (C,f)$ and $X = V^r (\Gamma,\pi)$.
\end{proof}

Another result of Welters is that, for a general \'{e}tale double cover, the singular locus of $V^r (C,f)$ is $V^{r+2} (C,f)$ \cite{Welters85}.  This is stronger than what can be deduced from Theorem~\ref{Thm:Lifting}, but we can show that $V^r (C,f)$ is smooth at the generic point.

\begin{corollary}
\label{Cor:Reduced}
Let $\Gamma$ be an $r$-generic loop of loops and let $\pi \colon \WG \to \Gamma$ be the quotient by the antipodal involution.  If $f \colon \WC \to C$ is an \'{e}tale double cover of curves specializing to $\pi$, then $V^r (C,f)$ is reduced.  In particular, if $g-1 = {{r+1}\choose{2}}$, then $V^r (C,f)$ is smooth.
\end{corollary}

\begin{proof}
In the proof of Theorem~\ref{Thm:Lifting}, we show that the intersection of $V^r (C,f)$ with ${{r+1}\choose{2}}$ general translates of $V^1 (C,f)$ consists of distinct reduced points.  Thus, $V^r (C,f)$ is smooth at these points and hence is generically reduced.  Since $V^r (C,f)$ is a type-D degeneracy locus, it is Cohen-Macaulay.  It follows that it has no embedded points, so it is reduced.
\end{proof}

We briefly discuss a connection to the theory of tropical linear series, as developed in \cite{TLS}.  A subset $\Sigma \subseteq R(D)$ is a \emph{tropical submodule} if it is closed under pointwise minimum and addition of scalars.  Just as one defines the Baker-Norine rank of a divisor $D$, one can similarly define the Baker-Norine rank of any tropical submodule $\Sigma \subseteq R(D)$.  Precisely, the \emph{Baker-Norine rank} of $\Sigma$ is the largest integer $r$ such that, for every effective divisor $E$ of degree $r$, there exists $\varphi \in \Sigma$ such that $\ddiv (\varphi) + D - E$ is effective.  The \emph{independence rank} of $\Sigma$ is the size of the largest tropically independent subset of $\Sigma$.  A \emph{tropical linear series} of rank $r$ is a finitely generated tropical submodule $\Sigma \subseteq R(D)$ of Baker-Norine rank $r$ and independence rank $r+1$.  In Corollary~\ref{Cor:TLS} below, we show that, for all $[D] \in V^r (\Gamma,\pi)$, there exists a tropical linear series $\Sigma \subseteq R(D)$ of rank $r$.  Furthermore, if $[D]$ is vertex avoiding, then $\Sigma$ contains the functions $\varphi_0 , \ldots , \varphi_r$ defined in Proposition~\ref{Prop:Fns}.

The literature contains several other definitions of linear series on tropical curves.  For example, there are the \emph{combinatorial limit linear series} of Amini and Gierczak \cite{AG24} and the \emph{matroidal linear series} of \cite{TLS}.  Another definition of tropical linear series appears in \cite[Definition~6.5]{M23}; these objects are referred to as \emph{strongly recursive tropical linear series} in \cite{TLS}.  The relationships between these various definitions are explored in \cite{Burkholder}.  The proof of Corollary~\ref{Cor:TLS} only uses properties that are common to all these definitions.

\begin{corollary}
\label{Cor:TLS}
Let $\Gamma$ be an $r$-generic loop of loops and $\pi \colon \WG \to \Gamma$ the quotient by the antipodal involution.  For every divisor $D \in V^r (\Gamma,\pi)$, there exists a tropical linear series $\Sigma \subseteq R(D)$ of rank $r$.  Moreover, if $[D]$ is vertex avoiding, then $\Sigma$ contains the functions $\varphi_0 , \ldots , \varphi_r$ defined in Proposition~\ref{Prop:Fns}.
\end{corollary}

\begin{proof}
The existence of the tropical linear series $\Sigma$ follows from Theorem~\ref{Thm:Lifting} because the tropicalization of a linear series is a tropical linear series.  Now, assume that $[D]$ is vertex avoiding, and let $i \in \{ 0, \ldots , r \}$.  For each $j \in J^-(i)$, let $p_j \in \gamma_j \smallsetminus \{ v_j , w_j \}$.  Let $E_i = \sum_{j \in J^-(i)} p_j$, and let $D_i = \ddiv (\varphi_i) + D - E_i$.

If $j \notin J(i)$, then $D_i \vert \gamma_j$ is a single point $q_j \notin \{ v_j , w_j \}$.  If $j \in J^+ (i)$, then $D_i \vert \gamma_j = 0$.  If $j \in J^- (i)$, then $D_i \vert \gamma_j = v_j + w_j - p_j$, which is equivalent to a unique point $q_j \in \gamma_j \smallsetminus \{ v_j , w_j \}$.  It follows that $D_i$ is equivalent to $D'_i = \sum q_j$.  By Dhar's burning algorithm, $D'_i$ is $q$-reduced for all points $q \in \WG$, hence $D'_i$ is the unique effective divisor equivalent to $D-E_i$.  Equivalently, there exists a unique function $\psi_i$ such that $\ddiv (\psi_i) + D - E_i$ is effective and $\psi_i (v_1) = 0$.  Since $\Sigma$ has Baker-Norine rank at least $r$ and is closed under tropical scaling, the function $\psi_i$ must be contained in $\Sigma$.  Since the set $\{ \psi \in \Sigma \mid \psi (v_1) = 0 \}$ is compact, it must also contain the limit of the functions $\psi_i$ as the points $p_j$ approach $v_j$.  This limit is $\varphi_i$, hence $\varphi_i \in \Sigma$ for all $i$.
\end{proof}

\section{The Prym-Brill-Noether Variety for Arbitrary Loops of Loops}
\label{Sec:Arbitrary}

\subsection{Displacement Prym Words}
\label{Sec:DispPrym}

In this section, we identify the Prym-Brill-Noether variety of a loop of loops with arbitrary edge lengths.  We first define some helpful words in the type A Coxeter group.

\begin{definition}
Let $m \geq 2$ be an integer and $i \in \Z/m\Z$.  Define the word 
\[
t_{m,i} := \prod_{\substack{j \in \{ 0, \ldots , r-1 \} \\ j \equiv i \Mod{m}}} \tau_j.
\]
We define $\overline{t}_{m,i} := t_{m,r-1-i}$.
\end{definition}

Let $(W,S)$ be the Coxeter group $A_r$ and recall that $S^*$ denotes the set of words in $(W,S)$.  To each divisor $D \in P^r$, we associate a function $s \colon \Z/(2g-2)\Z \to S^{\ast}$ as follows:
\[
s(j) := \left\{ \begin{array}{ll}
t_{m_j,i} &\text{if } \xi_j (D) \equiv s^r_j [i] \in \R/m_j\Z \\
\epsilon &\text{otherwise.}
\end{array} \right.
\]
If $s^r_j [i] \equiv s^r_j [i']$ in $\R/m_j\Z$, then $i \equiv i' \Mod{m_j}$, hence $s$ is well defined.
We write $Q_D$ for the lingering word obtained by concatenating $Q_D = s(1) s(2) \cdots s(2g-2)$.  As in Section~\ref{Sec:Generic}, the function $s$ has an important property.

\begin{definition}
\label{Def:PrymWord}
Let $\vec{m}$ be a torsion profile.  An \emph{$\vec{m}$-displacement Prym word} $Q$ of length $2g-2$ is a function $s \colon \Z/(2g-2)\Z \to S^{\ast}$ such that, for all $j \in \Z/(2g-2)\Z$, we have:
\begin{enumerate} 
\item either $s(j) = \epsilon$ or $s(j) = t_{m_j,i}$ for some $i \in \Z/m_j\Z$, and
\item $s(g-1+j) = \overline{s(j)}$. 
\end{enumerate}
\end{definition}

\begin{lemma}
\label{Lem:DispPrymWord}
If $D \in P^r$, then $Q_D$ is an $\vec{m}$-displacement Prym word.
\end{lemma}

\begin{proof}
The proof is exactly the same as that of Lemma~\ref{Lem:PrymWord}.
\end{proof}

As in Section~\ref{Sec:Generic}, we define the $j$th \emph{half word} of an $\vec{m}$-displacement Prym word to be \[Q^j = s(j) s(j+1) \cdots s(g-2+j).\]  An \emph{$\vec{m}$-displacement lingering word} of length $g-1$ is a function $s \colon [g-1] \to S^{\ast}$ such that, for all $j$, we have either $s(j) = \epsilon$ or $s(j) = t_{m_j,i}$ for some $i \in \Z/m_j\Z$.  The map $Q \mapsto Q^j$ sends an $\vec{m}$-displacement Prym word of length $2g-2$ to an $\vec{m}$-displacement lingering word of length $g-1$.  As in Section~\ref{Sec:Generic}, $\vec{m}$-displacement Prym words of length $2g-2$ are in bijection with $\vec{m}$-displacement lingering words of length $g-1$.

Given a $\vec{m}$-displacement lingering word $Q$ of length $g-1$ in the Coxeter group $A_r$, we define the set
\begin{align*}
T(Q) :&= \{ [D] \in P^r \mid Q \text{ is a lingering subword of } Q_D^1 \} \\
& = \{ [D] \in P^r \mid \xi_j (D) = s^r_j [i] \text{ whenever } s(j) = t_{m_j,i} \} .
\end{align*}
Note that the tori $T(Q)$ satisfy the same containment relations as in Lemma~\ref{Lem:Contains}.  The following is the main result of this section.

\begin{theorem}
\label{Thm:Arbitrary}
Let $\Gamma$ be a loop of loops with torsion profile $\vec{m}$ and $\pi \colon \WG \to \Gamma$ the quotient by the antipodal involution.  Then
\[
V^r (\Gamma,\pi) = \bigcup T(Q),
\]
where the union is over all $\vec{m}$-displacement lingering words of length $g-1$ that contain a reduced subword for $w_0$.
\end{theorem}

\begin{proof}
The proof is nearly identical to that of Theorem~\ref{Thm:Generic}, and we highlight the differences.  First, let $Q$ be an $\vec{m}$-displacement lingering word of length $g-1$ that contains a reduced subword $Q'$ for $w_0$, and let $[D] \in T(Q)$.  For each $j$, let $s'(j)$ be the subword of $s(j)$ that is contained in $Q'$.  Let $Q'_1 = s'(1) \cdots s'(g-1) s'(g) \cdots s'(2g-2)$, and for any $j$ in the range $1 \leq j \leq 2g-2$, let $\sigma_j = \prod_{i=1}^{j-1} s' (i)$.  We then construct functions $\varphi_0 , \ldots , \varphi_r \in R(D)$ such that the incoming slope of $\varphi_i$ at $v_j$ is $s^r_j [\sigma_j (i)]$ and the outgoing slope of $\varphi_i$ at $w_j$ is $s'^r_j[\sigma_j(i)]$.  The existence of such functions follows exactly as in Proposition~\ref{Prop:Fns}.  Because $\sigma_j$ is a permutation for all $j$, these functions satisfy the hypotheses of Corollary~\ref{Cor:HowToComputeRank}, hence $D$ has rank at least $r$.

For the reverse containment, let $[D] \in V^r (\Gamma,\pi)$.  Since $[D] \in T(Q^1_D)$, it suffices to show that every half word $Q^j_D$ contains a reduced subword for $w_0$.  By the same argument as in the proof of Theorem~\ref{Thm:Generic}, the word $Q_D$ has full rank, and the result then follows from Proposition~\ref{Prop:FullRank}.
\end{proof}

\begin{remark}
\label{Rmk:Contains}
Note that an $\vec{m}$-displacement lingering word $Q$ may contain a reduced subword for $w_0$ that is not itself an $\vec{m}$-displacement lingering word.  For example, fix an integer $k \geq 2$ and suppose that $m_j = k$ for all $j$.  If $k$ divides $r$ and $r/k$ is even, then every $\vec{m}$-displacement lingering word represents an even permutation.  If $r \equiv 2 \Mod{4}$, however, then $w_0$ is an odd permutation, so it cannot be represented by an $\vec{m}$-displacement lingering word.

For example, if $k=3$ and $r=6$, then the word
\[
Q = \Big( t_{3,1} t_{3,0} t_{3,2} \Big)^4 = \Big( \tau_1 \tau_4 \tau_0 \tau_3 \tau_2 \tau_5 \Big)^4 
\]
is an $\vec{m}$-displacement lingering word of length 12.  It contains the subword
\[
Q' = \Big( \tau_1 \tau_4 \tau_0 \tau_3 \tau_2 \tau_5 \Big)^3 \tau_4 \tau_3 \tau_5 ,
\]
which is a reduced subword for $w_0$, but the subword $Q'$ is not an $\vec{m}$-displacement lingering word.
\end{remark}

\subsection{The Prym-Brill-Noether Variety on a $k$-Uniform Loop of Loops}
\label{Sec:kUniform}

Let $k \geq 2$ be an integer.  Recall that a loop of loops $\Gamma$ is $k$-\emph{uniform} if $m_j = k$ for all $j$.  In this case, we say that an $\vec{m}$-displacement lingering word is a \emph{$k$-uniform lingering word}.  In this section, we will prove the following theorem.

\begin{theorem}
\label{Thm:kUniform}
Let $\Gamma$ be a $k$-uniform loop of loops and $\pi \colon \WG \to \Gamma$ the quotient by the antipodal involution.  Then
\[
\mathrm{dim} V^r (\Gamma,\pi) \leq \left\{ \begin{array}{ll} 
g-1-\frac{k(r+1)}{2} &\text{if } k<r\\
g-1-\frac{r(r+1)}{2} &\text{if } k \geq r,
\end{array} \right.
\]
with equality when $k$ is even.
\end{theorem}

In Section~\ref{Sec:Bielliptic} below, we will consider the special case where $k=2$.  In this case, we will show that $V^r (\Gamma,\pi)$ is pure dimensional, of dimension $r+1$.  When $k > 2$, we do not know that $V^r (\Gamma,\pi)$ is pure dimensional, and when $k$ is odd, we do not know that the above bound on $\mathrm{dim} V^r (\Gamma,\pi)$ is optimal.

When $k \geq r$, the $k$-uniform loop of loops is $r$-generic, and Theorem~\ref{Thm:kUniform} follows from Corollary~\ref{Cor:PureDimension}.  When $k<r$, by Theorem~\ref{Thm:Arbitrary}, it suffices to show that every $k$-uniform lingering word that contains a reduced subword for $w_0$ has length at least $\frac{k(r+1)}{2}$.  To prove this, we first require the following technical lemma.

\begin{lemma}
\label{Lem:EndCount}
Let $Q = s_1 s_2 \cdots s_{\ell}$ be a reduced word in $A_r$ for $w_0$.  Then, for any nonnegative integer $n < \frac{r-1}{2}$, we have
\[
\# \{ j \mid s_j = \tau_i \text{ such that } n \leq i \leq r-1-n \} \geq (\frac{r}{2}-n)(r+1).
\]
\end{lemma}

\begin{proof}
For $a \in \{ 0, \ldots , r \}$, let
\[
B(n,a) = \{ j \in J(a) \mid s_j  = \tau_i \text{ such that } n \leq i \leq r-1-n \}.
\]
We claim that $\# B(n,a) \geq r-2n$.  The result follows from this by summing over all $a$ because, for every $j$, there are exactly two values of $a$ such that $j \in J(a)$.

To see the claim, first note that if $B(n,a) = J(a)$, then $\# B(n,a) = \# J(a) = r$, and the result follows.  Otherwise, there exists a $j \in J(a)$ such that either $\sigma_j (a) \leq n$ or $\sigma_j (a) \geq r-n$.  If $\sigma_j (a) \leq n$, then the subword of $Q$ consisting of the terms $s_{j'}$ with $j' \in B(n,a)$ and $j' \leq j$ must contain the subword $\tau_{a-1} \tau_{a-2} \cdots \tau_n$, which has length $a-n$.  The subword of $Q$ consisting of the terms $s_{j'}$ with $j' \in B(n,a)$ and $j' > j$ must contain the subword $\tau_n \tau_{n+1} \cdots \tau_{r-a-1}$, which has length $r-a-n$.  Thus, $\# B(n,a) \geq (a-n) + (r-a-n) = r-2n$.

Similarly, if $\sigma_j (a) \geq n$, then the subword of $Q$ consisting of the terms $s_{j'}$ with $j' \in B(n,a)$ and $j' \leq j$ must contain the subword $\tau_a \tau_{a+1} \cdots \tau_{r-1-n}$, which has length $r-a-n$.  The subword of $Q$ consisting of the terms $s_{j'}$ with $j' \in B(n,a)$ and $j' > j$ must contain the subword $\tau_{r-1-n} \tau_{r-2-n} \cdots \tau_{r-a}$, which has length $a-n$.  Thus again, $\# B(n,a) \geq (a-n) + (r-a-n) = r-2n$.
\end{proof}

\begin{proof}[Proof of Theorem~\ref{Thm:kUniform}]
Let $Q$ be a $k$-uniform word that contains a reduced subword $Q'$ for $w_0$.  For $k < r$, it suffices to prove that $Q$ has length at least $\frac{k(r+1)}{2}$.  When $r$ is divisible by $k$, this is straightforward -- each word $t_{k,i}$ has length $r/k$, and $Q'$ has length $\frac{r(r+1)}{2}$, so $Q$ must have length at least $\frac{k(r+1)}{2}$.  More generally, by the division algorithm, we may write $r = qk + \rho$, where $0 \leq \rho \leq k-1$.  For an integer $i$ in the range $0 \leq i \leq k-1$, the word $t_{k,i}$ has length $q+1$ if $i \leq \rho$ and length $q$ if $i > \rho$.

Let $Q' = s_1 s_2 \cdots s_{\ell}$, where $\ell = {{r+1}\choose{2}}$.  First, suppose that $\rho$ is even.  By Lemma~\ref{Lem:EndCount}, we have
\[
\# \{ i \mid s_i = \tau_j \text{ such that } \frac{\rho}{2} \leq j \leq r-1-\frac{\rho}{2} \} \geq \frac{r-\rho}{2}(r+1).
\]
For each $i$, the number of reflections $\tau_j$ occuring in $t_{k,i}$ such that $\frac{\rho}{2} \leq j \leq r-1-\frac{\rho}{2}$ is exactly $q$.  It follows that the length of $Q$ must be at least
\[
\frac{(r-\rho)(r+1)}{2q} = \frac{k(r+1)}{2}.
\]

Similarly, if $\rho$ is odd, then by Lemma~\ref{Lem:EndCount} we have either
\begin{align*}
\# \{ i \mid s_i = \tau_j \text{ such that } \frac{\rho+1}{2} \leq j \leq r-1-\frac{\rho-1}{2} \} &\geq \frac{r-\rho}{2}(r+1)\text{ or } \\
\# \{ i \mid s_i = \tau_j \text{ such that } \frac{\rho-1}{2} \leq j \leq r-1-\frac{\rho+1}{2} \} &\geq \frac{r-\rho}{2}(r+1).
\end{align*}
In either case, as above, the number of reflections $\tau_j$ occuring in $t_{k,i}$ with $j$ in the appropriate range is exactly $q$.  So again, the length of $Q$ must be at least 
\[
\frac{(r-\rho)(r+1)}{2q} = \frac{k(r+1)}{2}.
\]

Finally, assume that $k$ is even.  We show that there exists a $k$-uniform word representing $w_0$ of length $\frac{k(r+1)}{2}$.  For $j \in \Z/2\Z$, define the $k$-uniform word $u_j = \prod_{i \equiv j \Mod{2}} t_{k,i}$.  Then, if $r$ is odd, the word $(u_0 u_1)^{\frac{r+1}{2}}$ is a $k$-uniform word representing $w_0$.  This word has length $\frac{k(r+1)}{2}$ if $k \leq r$ and length $\frac{r(r+1)}{2}$ if $k \geq r$.  Similarly, if $r$ is even, the word $(u_0 u_1)^{\frac{r}{2}} u_0$ is a $k$-uniform word representing $w_0$.  This word has length $\frac{k(r+1)}{2}$ if $k \leq r$ and length $\frac{r(r+1)}{2}$ if $k \geq r$.  Thus, when $k$ is even, the bound on $\mathrm{dim} V^r (\Gamma, \pi)$ is an equality.
\end{proof}

\begin{proof}[Proof of Theorem~\ref{Thm:MainThm}]
Let $\Gamma$ be a $k$-uniform loop of loops, and assume further that $k \cdot \ell(v_j,w_j) = \ell (\gamma_j)$.  Let $\Sigma$ be a loop with vertices $x_j , y_j$ for $j \in \Z/(g-1)\Z$, where $x_j$ is adjacent to $y_{j-1}$ and $y_j$ for all $j$.  Define the length of the edge connecting $y_{j-1}$ to $x_j$ to be $\frac{\ell(\beta_j)}{k}$, and the length of the edge connecting $x_j$ to $y_j$ to be $\frac{\ell(\gamma_j)}{k}$.  Define a finite harmonic morphism of metric graphs $\psi \colon \Gamma \to \Sigma$ by $\psi (v_j) = x_j$, $\psi (w_j) = y_j$, with expansion factor $k$ along each edge $\beta_j$, and expansion factors $1, k-1$ on the two edges comprising $\gamma_j$ (see Figure~\ref{Fig:MapToElliptic}).  This morphism has degree $k$.  If the residue field has characteristic zero or relatively prime to $2k$, then this morphism if tame.  By \cite[Lemma~7.15]{ABBR15a}, there is a map of smooth curves $C \to E$ specializing to $\psi$, where $C$ has genus $g$ and $E$ has genus 1.  By \cite[Lemma~5.9]{JensenLen18}, the map $\pi \colon \WG \to \Gamma$ can be lifted to an unramified double cover $f \colon \WC \to C$.

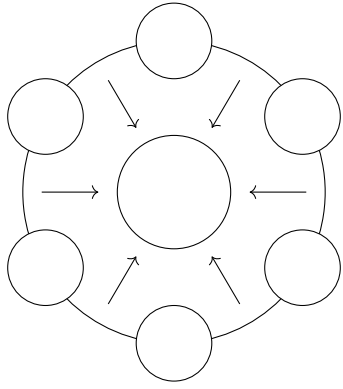
\begin{figure}[H]
\begin{tikzpicture}

\draw (0,0) circle (2);
\draw [fill=white] (0,2) circle (0.5);
\draw [fill=white] (1.7,1) circle (0.5);
\draw [fill=white] (1.7,-1) circle (0.5);
\draw [fill=white] (0,-2) circle (0.5);
\draw [fill=white] (-1.7,-1) circle (0.5);
\draw [fill=white] (-1.7,1) circle (0.5);
\draw (0,0) circle (0.75);
\draw (1.75,0) edge[->] (1,0);
\draw (0.87,1.48) edge[->] (0.5,0.85);
\draw (-0.87,1.48) edge[->] (-0.5,0.85);
\draw (-1.75,0) edge[->] (-1,0);
\draw (-0.87,-1.48) edge[->] (-0.5,-0.85);
\draw (0.87,-1.48) edge[->] (0.5,-0.85);

\end{tikzpicture}
\caption{The harmonic morpshim $\psi \colon \Gamma \to \Sigma$.}
\label{Fig:MapToElliptic}
\end{figure}

We show that the map $f \colon C \to E$ does not factor through a nontrivial isogeny.  Suppose there exists a curve $E'$ of genus 1, a map $f' \colon C \to E'$, and an isogeny $\iota \colon E' \to E$ of degree $d$ such that $f = \iota \circ f'$.  Then there exists a skeleton $\Sigma'$ of $E'$ and maps $\psi' \colon \Gamma \to E'$, $\nu \colon \Sigma' \to \Sigma$ such that $f'$ specializes to $\psi'$ and $\iota$ specializes to $\nu$.  Now, let $p \in \gamma_j$ be a point in the interior of the longer edge.  Then there exists a point $q \in \gamma_j$ on the shorter edge such that $\psi^* (\psi(p)) = p + (k-1)q$.  It follows that $\iota^{-1} (\psi (p))$ consists of at most 2 points.  If $\iota^{-1} (\psi (p))$ is a single point $p'$, then $\iota^* (\psi (p)) = dp'$, hence $\psi^* (\psi(p))$ is divisible by $d$.  But this implies that $d=1$, so $\iota$ is trivial.  Otherwise, suppose that $\iota^{-1} (\psi (p)) = \{ p',q' \}$, where $\psi' (p) = p'$ and $\psi' (q) = q'$.  Then $\psi^* (\psi(p))$ is divisible by $\deg (\psi')$, so $\psi'$ has degree 1, which is impossible.

By Baker's specialization lemma \cite[Corollary~2.11]{Baker08}, we have
\[
\Trop V^r (\mathcal{X},f) \subseteq V^r (\Gamma,\pi).
\]
If $k < r$, then by Gubler's Bieri-Groves theorem for maximally degenerate abelian varieties, we have
\[
\mathrm{dim} V^r (\mathcal{X},f) \leq \mathrm{dim} V^r (\Gamma,\pi) \leq g-1-\frac{k(r+1)}{2},
\]
where the second inequality holds by Theorem~\ref{Thm:kUniform}.  If $k \geq r$, then $\Gamma$ is $r$-generic, and the result holds by Corollary~\ref{Cor:Dimension}.
\end{proof}

\subsection{The Bielliptic Case}
\label{Sec:Bielliptic}

In the special case where $k=2$, we can say more.

\begin{theorem}
\label{Thm:Bielliptic}
Let $\Gamma$ be a 2-uniform loop of loops and $\pi \colon \WG \to \Gamma$ the quotient by the antipodal involution.  Then
\[
V^r (\Gamma,\pi) = \bigcup T(Q),
\]
where the union is over all lingering words of length $g-1$ in the Coxeter group $I_2 (r+1)$ that contain a reduced subword for $w_0$.
\end{theorem}

\begin{proof}
We write $r_0$ for the element of $S_{r+1}$ represented by the word $t_{2,0}$, and $r_1$ for the element represented by $t_{2,1}$.  We first show that the subgroup of $S_{r+1}$ generated by $r_0$ and $r_1$ is isomorphic to the dihedral group of order $2r+2$.  To see this, consider a regular polygon with $r+1$ vertices.  Label the vertices of the polygon with the integers $0, \ldots , r$ by first choosing a vertex to label 0, and then, starting with the vertex labeled 0 and proceeding clockwise, label the vertices in order with even integers, and proceeding counterclockwise, label the vertices in order with odd integers, as in Figure~\ref{Fig:Dihedral}.  The action of the dihedral group on these vertices induces an embedding of the dihedral group into $S_{r+1}$.  Under this embedding, reflection across the apothem through $\overline{01}$ yields the permutation $r_0$, whereas reflection across the radius through 0 yields the permutation $r_1$.  (Again, see Figure~\ref{Fig:Dihedral}.)  Since these two reflections generate the dihedral group, the result follows.

\begin{figure}[H]
\begin{tikzpicture}

\draw (2,0)--(1.4,1.4);
\draw (1.4,1.4)--(0,2);
\draw (0,2)--(-1.4,1.4);
\draw (-1.4,1.4)--(-2,0);
\draw (-2,0)--(-1.4,-1.4);
\draw (-1.4,-1.4)--(0,-2);
\draw (0,-2)--(1.4,-1.4);
\draw (1.4,-1.4)--(2,0);
\draw (2.2,0) node {$0$};
\draw (1.6,1.6) node {$1$};
\draw (0,2.2) node {$3$};
\draw (-1.6,1.6) node {$5$};
\draw (-2.2,0) node {$7$};
\draw (-1.6,-1.6) node {$6$};
\draw (0,-2.2) node {$4$};
\draw (1.6,-1.6) node {$2$};
\draw[dashed] (2.24,1.16)--(-2.24,-1.16);
\draw[dashed] (2.4,0)--(-2.4,0);

\end{tikzpicture}
\caption{An octagon labeled with the integers $0, \ldots 7$, as described in the paragraph above.  The permuatations $r_0$ and $r_1$ correspond to reflection across the two dashed lines, and the permutation $w_0$ corresponds to the unique rotation of order 2.}
\label{Fig:Dihedral}
\end{figure}
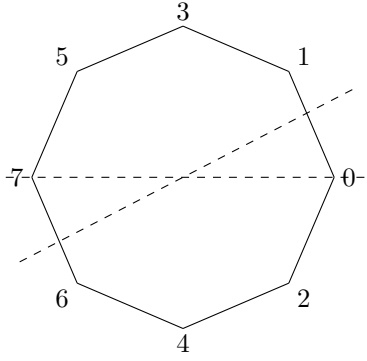

Now, let $W$ be the dihedral group and $S = \{ r_0 , r_1 \}$.  Then $(W,S)$ is the Coxeter group $I_2 (r+1)$.  If $r$ is even, then reflection across the radius through $\frac{r}{2}$ is both the permuation $w_0$ of maximal length in $A_r$ and the element of maximal length in $I_2 (r+1)$.  Similarly, if $r$ is odd, then the unique rotation of order 2 is the element of maximal length in both Coxeter groups.

Let $Q$ be a lingering 2-uniform word that contains a reduced subword $Q'$ for $w_0$.  By Theorem~\ref{Thm:Arbitrary}, it suffices to show that the corresponding word in $I_2 (r+1)$ also contains a reduced subword for $w_0$.  Without loss of generality, assume that the first term of $Q'$ is $\tau_i$ with $i$ odd.  (The case where $i$ is even follows by an essentially identical argument.)  We may then write $Q'$ uniquely as a concatenation of subwords
\[
Q' = Q'_1 Q'_2 \cdots Q'_{\ell}
\]
where $Q'_i$ is a word in the set $\{ \tau_j \mid j \equiv i \Mod{2} \}$.  Since $Q'$ is a subword of $Q$ and $Q$ is 2-uniform, it follows that $Q$ contains the subword $(\tau_{2,1} \tau_{2,0})^{\frac{\ell}{2}}$ if $\ell$ is even or the subword $(\tau_{2,1} \tau_{2,0})^{\frac{\ell-1}{2}} \tau_{2,1}$ if $\ell$ is odd.  Now, note that if $i \equiv j \Mod{2}$, then $\tau_i$ and $\tau_j$ commute.  Since $Q'$ is reduced, it follows that the terms appearing in each subword $Q'_i$ are distinct.  In other words, the length of $Q'_i$ is at most $\lceil \frac{r}{2} \rceil$ if $i$ is even and $\lfloor \frac{r}{2} \rfloor$ if $i$ is odd.  Since $Q'$ has length ${{r+1}\choose{2}}$, it follows that $\ell \geq r+1$.  Thus, the corresponding word in $I_2 (r+1)$ contains the subword $(r_1 r_0)^{\frac{r+1}{2}}$ if $r$ is odd or the subword $(r_1 r_0)^{\frac{r}{2}}r_1$ if $r$ is even.  But these are exactly the reduced words for $w_0$ in $I_2 (r+1)$ whose first term is $r_1$.
\end{proof}

Theorem~\ref{Thm:Bielliptic} implies that, when $\Gamma$ is 2-uniform, the Prym-Brill-Noether variety $V^r (\Gamma,\pi)$ exhibits many of the same topological properties as when $\Gamma$ is $r$-generic.

\begin{proposition}
\label{Prop:BiellipticPureDimension}
Let $\Gamma$ be a 2-uniform loop of loops and $\pi \colon \WG \to \Gamma$ the quotient by the antipodal involution.  Then $V^r (\Gamma,\pi)$ is pure dimensional, of dimension $g-1-(r+1)$.
\end{proposition}

\begin{proof}
By Theorem~\ref{Thm:Bielliptic}, we have $V^r (\Gamma,\pi) = \cup T(Q)$, where the union is over all elements of $P(I_2(r+1),w_0,g-1)$.  By Lemma~\ref{Lem:Contains}, we have $T(Q) \subseteq T(Q')$ if and only if $Q \leq Q'$ in $P(I_2(r+1),w_0,g-1)$.  Moreover, the dimension of $T(Q)$ is $\rho(Q)$.  It follows from Lemma~\ref{Lem:PosetIsPureDimension} that, if $T(Q)$ is maximal with respect to containment, then
\[
\mathrm{dim} T(Q) = \rho(Q) = g-1-(r+1).
\]
\end{proof}

\begin{proposition}
\label{Prop:BiellipticConnected}
Let $\Gamma$ be a 2-uniform loop of loops and $\pi \colon \WG \to \Gamma$ the quotient by the antipodal involution.  If $g-1 > r+1$, then $V^r (\Gamma,\pi)$ is connected in codimension 1.
\end{proposition}

\begin{proof}
If $T(Q)$ and $T(Q')$ are maximal, then by Lemma~\ref{Lem:BraidRelations}, there exists a sequence
\[
Q = Q_0 , Q_1 , \ldots , Q_m = Q'
\]
where, $T(Q_{i-1}) \cap T(Q_i) = T(Q_{i-1} \wedge Q_i)$ has codimension 1 in both $T(Q_{i-1}$ and $T(Q_i)$.
\end{proof}

\begin{lemma}
\label{Lem:BiellipticCardinality}
Let $\Gamma$ be a 2-uniform loop of loops and $\pi \colon \WG \to \Gamma$ the quotient by the antipodal involution.  If $g-1= r+1$, then $\# V^r (\Gamma,\pi) = 2$.
\end{lemma}

\begin{proof}
If $g-1= r+1$, then $\# V^r (\Gamma,\pi)$ is equal to the number of reduced words in $I_2(r+1)$ for $w_0$.  There are exactly two of these.  Specifically, if $r$ is odd, then the 2 reduced words for $w_0$ are $(r_1 r_0)^{\frac{r+1}{2}}$ and $(r_0 r_1)^{\frac{r+1}{2}}$.  If $r$ is even then the 2 reduced words for $w_0$ are $(r_1 r_0)^{\frac{r}{2}}r_1$ and $(r_0 r_1)^{\frac{r}{2}}r_0$.
\end{proof}

\subsection{The $k$-Gonal Case}
\label{Sec:kGonal}

\begin{proof}[Proof of Theorem~\ref{Thm:KGonal}]
By \cite{LenUlirsch21, CLRW20}, we know that for a general \'{e}tale double cover $f \colon \WC \to C$, we have
\[
\mathrm{dim} V^r (C,f) \leq g-1-n(r,k).
\]
A straightforward calculation shows that $n(r,k) \leq \frac{(k-1)(r+1)}{2}$ when $r+1 < \ell (\ell +1)$, so it suffices to show that, if $k < r+1$, then 
\[
\mathrm{dim} V^r (C,f) \leq g-1-\frac{(k-1)(r+1)}{2}.
\]

To see this, let $\Gamma$ be a $(k-1)$-uniform loop of loops, with edge lengths as in the proof of Theorem~\ref{Thm:MainThm}.  Assume further that the length of $\beta_1$ is greater than the sum of all other edge lengths combined.  Then the divisor $kv_1$ has rank at least 1.  Moreover, this divisor is the pullback of a point via a harmonic morphism to an interval.  This harmonic morphism has dilation factor $k-1$ along each edge $\beta_j$ for $j \neq 1$, dilation factors $1, k-2$ along the two edges comprising $\gamma_j$, dilation factor 1 along a subinterval of $\beta_1$ with endpoint $v_1$, and dilation factor $k-1$ along the remainder of $\beta_1$.  (See Figure~\ref{Fig:MapToInterval}.)  If the residue field has characteristic zero or relatively prime to $2k$, then this morphism if tame.  By \cite[Lemma~7.15]{ABBR15a}, there is a map of smooth curves $C \to \PP^1$ specializing to this moprhism, where $C$ has genus $g$.  As in the proof of Theorem~\ref{Thm:MainThm}, by \cite[Lemma~5.9]{JensenLen18}, the map $\pi \colon \WG \to \Gamma$ can be lifted to an unramified double cover $f \colon \WC \to C$.

\begin{figure}[H]
\begin{tikzpicture}

\draw (0,0) circle (0.5);
\draw (2,0) circle (0.5);
\draw (4,0) circle (0.5);
\draw (6,0) circle (0.5);
\draw (6.5,0)--(7.5,0);
\draw (0.5,0)--(1.5,0);
\draw (2.5,0)--(3.5,0);
\draw (4.5,0)--(5.5,0);
\draw (-0.5,0) to [out=-90, in=240] (7.5,0);

\draw (-0.5,-3.5)--(7.5,-3.5);
\draw (3,-2.25) edge[->] (3,-3.25);

\end{tikzpicture}
\caption{A harmonic morpshim from $\Gamma$ to an interval.}
\label{Fig:MapToInterval}
\end{figure}
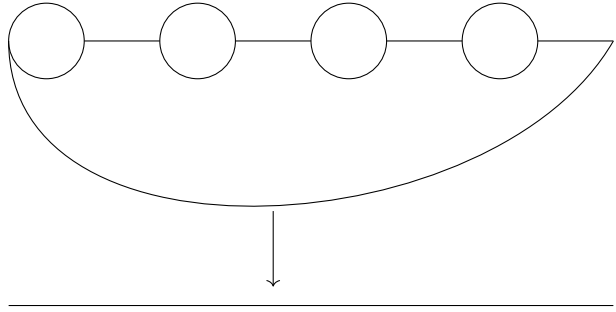

As in the proof of Theorem~\ref{Thm:MainThm}, by \cite[Corollary~2.11]{Baker08}, we have
\[
\Trop V^r (C,f) \subseteq V^r (\Gamma,\pi).
\]
If $k < r+1$, then by Gubler's Bieri-Groves theorem for maximally degenerate abelian varieties, we have
\[
\mathrm{dim} V^r (C,f) \leq \mathrm{dim} V^r (\Gamma,\pi) \leq g-1-\frac{(k-1)(r+1)}{2},
\]
where the second inequality holds by Theorem~\ref{Thm:kUniform}.  If $k \geq r+1$, then $\Gamma$ is $r$-generic, and the result holds by Corollary~\ref{Cor:Dimension}.
\end{proof}

\bibliography{math}

\end{document}